\documentclass[a4paper]{amsart}
\usepackage{latexsym,bm,stmaryrd}
\usepackage{amsmath,amsthm,amsfonts,amssymb,mathrsfs,pb-diagram}

\usepackage{microtype}
\usepackage{mathtools}
\usepackage{mathbbol,wasysym}

\let\<=\langle
\let\>=\rangle
\def\({\big(}
\def\){\big)}

\def\Z{\mathbb{Z}}
\def\Q{\mathbb{Q}}

\def\N{\mathbb{N}}

\def\t{\mathfrak{t}}

\def\O{\mathcal{O}}

\def\lam{\lambda}
\def\Lam{\Lambda}
\def\Sym{\mathfrak{S}}
\def\eps{\varepsilon}
\def\bc{\mathbf{c}}
\newcommand\RR{\mathscr{R}}
\newcommand\R[1][n]{\RR_{#1}^{\Lambda}}

\def\bn[#1,#2]{\begin{bmatrix}#1\\#2\end{bmatrix}}

\def\fg{\mathfrak{g}}

\newcommand\bi{\mathbf{i}}

\def\veps{\varepsilon}

\DeclareMathOperator\Hom{Hom}

\DeclareMathOperator\cha{char}

\DeclareMathOperator\Tr{Tr}

\DeclareMathOperator\PD{\mathscr{P\!D}}

\DeclareMathOperator\diag{diag}

\DeclareMathOperator\wt{wt}
\DeclareMathOperator\Proj{Proj}

\DeclareMathOperator\Mod{Mod}
\DeclareMathOperator\NH{NH}
\DeclareMathOperator\Irr{Irr}

\title[]
{Piecewise dominant sequences and the cocenter of the cyclotomic quiver Hecke algebras}
\subjclass[2010]{20C08, 16G99, 06B15}
\keywords{Cyclotomic quiver Hecke algebras, categorification}
\author{Jun Hu}\address{MIIT Key Laboratory of Mathematical Theory and Computation in Information Security\\
School of Mathematics and Statistics\\
  Beijing Institute of Technology\\
  Beijing, 100081, P.R. China}
\email{junhu404@bit.edu.cn}
\author{Lei Shi}\address{School of Mathematics and Statistics\\
  Beijing Institute of Technology\\
  Beijing, 100081, P.R. China}
\email{3120195738@bit.edu.cn}

\numberwithin{equation}{section}
\newtheorem{prop}[equation]{Proposition}
\newtheorem{thm}[equation]{Theorem}

\newtheorem{cor}[equation]{Corollary}
\newtheorem{conj}[equation]{Conjecture}
\newtheorem{conja}[equation]{Center Conjecture}
\newtheorem{conjb}[equation]{Indecomposability Conjecture}

\newtheorem{lem}[equation]{Lemma}
\newtheorem{examp}[equation]{Example}

\theoremstyle{definition}
\newtheorem{dfn}[equation]{Definition}
\theoremstyle{remark}
\newtheorem{rem}[equation]{Remark}

\begin{document}

\begin{abstract} In this paper we study the cocenter of the cyclotomic quiver Hecke algebra $\RR^\Lam_\alpha$ associated to an {\it arbitrary} symmetrizable Cartan matrix $A=(a_{ij})_{i,j}\in I$, $\Lam\in P^+$ and $\alpha\in Q_n^+$. We introduce a notion called  ``piecewise dominant sequence'' and use it to construct  some explicit homogeneous elements which span the maximal degree component of the cocenter of $\RR^\Lam_\alpha$. We show that the minimal degree components of the cocenter of $\RR^\Lam_\alpha$ is spanned by the image of some KLR idempotent $e(\nu)$, where each $\nu\in I^\alpha$ is piecewise dominant. As an application, we show that the weight space $L(\Lam)_{\Lam-\alpha}$ of the irreducible highest weight module $L(\Lam)$ over $\mathfrak{g}(A)$ is nonzero (equivalently, $\R[\alpha]\neq 0$) if and only if there exists a piecewise dominant sequence $\nu\in I^\alpha$.
\end{abstract}

\maketitle
\setcounter{tocdepth}{1}
\tableofcontents

\section{Introduction}

The quiver Hecke algebras (also known as the KLR algebras) $\RR_\beta$ are some remarkable infinite families of $\Z$-graded algebras introduced by Khovanov-Lauda \cite{KL1, KL2}, and independently by Rouquier \cite{Rou1, Rou2} around 2008. These algebras depend on a symmetrizable Cartan matrix $A=(a_{ij})_{i,j\in I}$,
$\beta\in Q_n^+$ and some polynomials $\{Q_{ij}(u,v)|i,j\in I\}$. These algebras play important roles in the categorification of the negative part $U_q(\mathfrak{g})^{-}$ of the quantum group $U_q(\mathfrak{g})$, where $\mathfrak{g}=\mathfrak{g}(A)$. When the ground field $K$ has characteristic $0$, $A$ is symmetric and the polynomial $\{Q_{ij}(u,v)|i,j\in I\}$ are chosen as \cite[\S3.2.4]{Rou2}, Rouquier \cite{Rou2}, and independently Varagnolo-Vasserot \cite{VV} have proved that the categorification sends the indecomposable projective modules over the quiver Hecke algebra $\RR_\beta$ to the canonical bases of $U_q(\mathfrak{g})^{-}$.

For each dominant integral weight $\Lam\in P^+$, Khovanov-Lauda, and Rouquier have also introduced a graded quotient $\R[\beta]$ of $\RR_\beta$, called the cyclotomic quiver Hecke algebra (also known as the cyclotomic KLR algebras). Khovanov and Lauda have conjectured that the category of finite dimensional projective modules over these $\R[\beta]$ should give a categorification of the integrable highest weight module $L(\Lam)$ over the quantum group $U_q(\mathfrak{g})$. Khovanov-Lauda's Cyclotomic Categorification Conjecture was later proved by Kang and Kashiwara \cite{KK}. The cyclotomic quiver Hecke algebra $\R[\beta]$ behaves in many aspects similar to the cyclotomic Hecke algebra of type $G(\ell,1,n)$ (also known as the Ariki-Koike algebras). In fact, in the case of type $A_{\infty}$ or affine type $A_{(e-1)}^{(1)}$, assuming the ground field $K$ contains a primitive $e$-th root of unity, and $\{Q_{ij}(u,v)|i,j\in I\}$ are chosen as \cite[\S3.2.4]{Rou2}, Brundan and Kleshchev have proved in \cite{BK:GradedKL} that each $\R[\beta]$ is isomorphic to the block algebra of the cyclotomic Hecke algebra of type $G(\ell,1,n)$ corresponding to $\beta$, where $\ell$ is the level of $\Lam$. Mathas and the fist author of this paper have constructed a $\Z$-graded cellular basis for the cyclotomic quiver Hecke algebra $\R[\beta]$, and use this basis to construct a homogeneous symmetrizing form on $\R[\beta]$ (see \cite{HM}, \cite[Remark 4.7]{HuL}). More recently, Mathas and Tubenhauer have constructed $\Z$-graded cellular bases for the cyclotomic quiver Hecke algebra in types $C_{\Z_{\geq 0}}, C_e^{(1)}, B_{\Z_{\geq 0}}, A_{2e}^{(2)}$ and $D_{e+1}^{(2)}$.

For the cyclotomic quiver Hecke algebras $\R[\beta]$ of general type, we have given in \cite{HS} a closed formula for the graded dimension of $\R[\beta]$. The formula depends only on the root system associated to $A$ and the dominant weight $\Lam$ but not on the chosen ground field $K$, which immediately implies that the cyclotomic quiver Hecke algebra $\R[\beta](\O)$ is free over $\O$ for any commutative ground ring $\O$. These graded dimension formulae are also generalized to the cyclotomic quiver Hecke superalgebras in \cite{HS2}. The $i$-restriction functor $E_i$ and the $i$-induction functor $F_i$ play key roles in Kang and Kashiwara's proof of Khovanov-Lauda's Cyclotomic Categorification Conjecture. Rouquier has noticed (\cite{Rou1}) that the biadjointness of $E_i$ and $F_i$ induces a natural homogeneous Frobenius form on $\R[\beta]$. Shan, Varagnolo and Vasserot have proved (\cite{SVV}) that this Frobenius form is actually a homogeneous symmetrizing form on $\R[\beta]$ of degree $-d_{\Lam,\beta}$. There are now two major unsolved open problems on $\R[\beta]$ as follows:

\begin{conja} The centers of $\R[\beta]$ consists of symmetric elements in its KLR $x$ and $e(\nu)$ generators.
\end{conja}

\begin{conjb}\label{conj2}  The algebra $\R[\beta]$ is indecomposable.
\end{conjb}

The second conjecture is equivalent to the claim that the dimension of the degree $0$ component of the center $Z(\R[\beta])$ is one dimensional (i.e., equal to $Ke(\beta)$). In particular, the Indecomposibility Conjecture is actually a consequence of the Center Conjecture. Conjecture \ref{conj2} is a slight generalization of \cite[Conjecture 3.33]{SVV} in that $K$ can be of positive characteristic. There are a number of special cases where the above two conjectures were verified. For example, assume $\{Q_{ij}(u,v)|i,j\in I\}$ are given as \cite[(11)]{SVV}. If $\cha K=0$, $\mathfrak{g}$ is symmetric and of finite type, and  then the above conjecture holds by \cite[Remark 3.41]{SVV}; if $\mathfrak{g}$ be of type $A_{\infty}$ or affine type $A_{(e-1)}^{(1)}$ with $e>1$ and $(e,p)=1$, where $p:=\cha K$, then the main result of \cite{BK:GradedKL} shows that over a finite extension field of $K$, each cyclotomic quiver Hecke algebra $\R[\alpha]$ is isomorphic to the block algebra of the cyclotomic Hecke algebra of type $G(\ell,1,n)$ (\cite{BK:GradedKL}) which corresponds to $\alpha$. In this case, the above conjecture holds because $e(\alpha):=\sum_{i\in I^\alpha}e(\bi)$ is a block idempotent of the corresponding cyclotomic Hecke algebra by \cite{LM} and \cite{Br}. By \cite[Theorem 1.9]{HS2} and \cite{HuLin}, for arbitrary $\{Q_{ij}(u,v)|i,j\in I\}$, the above conjecture also holds whenever $\beta=\sum_{j=1}^n\alpha_{i_j}$ with $\alpha_{i_1},\cdots, \alpha_{i_n}$ pairwise distinct. The first author and Huang Lin proposed in \cite{HuLin} a ``cocenter approach'' to the Center Conjecture for the cyclotomic quiver Hecke algebra $\R[\beta]$ of general type.

The current work is motivated by the study of the above two conjectures and the ``cocenter approach'' proposed in \cite{HuLin}. The degree $-d_{\Lam,\beta}$ homogeneous symmetrizing form on $\R[\beta]$ implies that there is a $\Z$-graded linear isomorphism: $\Tr(\R[\beta]):=\R[\beta]/[\R[\beta],\R[\beta]]\cong \bigl(Z(\R[\beta])\bigr)^*\<d_{\Lam,\beta}\>$. Thus the study of the degree $0$ component of the center $Z(\R[\beta])$ can be transformed into the study of the degree $d_{\Lam,\beta}$ component of the cocenter $\Tr(\R[\beta])$. In this paper we introduce in Definition \ref{pddfn} a new notion called  ``piecewise dominant sequences'', and use them to construct in Theorem \ref{principle generator} some explicit homogeneous elements which span the  degree $d_{\Lam,\beta}$ component of the cocenter of $\RR^\Lam_\alpha$. Remarkably, these generators are all polynomials in the KLR's x generators. We show in Theorem \ref{principle generator} that the degree $0$ components of the cocenter of $\RR^\Lam_\alpha$ is spanned by the image of $e(\nu)$ for those piecewise dominant sequence $\nu$ in $I^\alpha$. As an application, we show that the weight space $L(\Lam)_{\Lam-\alpha}$ of the irreducible highest weight module $L(\Lam)$ over $\mathfrak{g}(A)$ is nonzero (equivalently, $\R[\alpha]\neq 0$) if and only if there exists a piecewise dominant sequence $\nu\in I^\alpha$. We give some normal form for the $K$-linear generators of the cocenter $\Tr(\R[\alpha])$ in Theorem \ref{generator}. We recover in Corollary  \ref{maincor0} Shan-Varagnolo-Vasserot's result \cite[Theorem 3.31(a)]{SVV} on the range of the degrees of the cocenter $\Tr(\R[\beta])$ in an elementary way. We also show in Theorem \ref{4mainthm} that Conjecture \ref{conj2} holds over {\it arbitrary} ground field when $\{Q_{ij}(u,v)|i,j\in I\}$ are given as \cite[\S3.2.4]{Rou2}, $\mathfrak{g}$ is either symmetric and of finite type, or $\mathfrak{g}$ is of type $A_{\infty}$ or affine type $A_{(e-1)}^{(1)}$ with $e>1$.
\smallskip

The paper is organised as follows. In Section 2 we recall some basic definitions and properties of the quiver Hecke algebra $\RR_\beta$ and its cyclotomic
quotient $\R[\beta]$. In Section 3 we investigate some relations inside the cocenter $\Tr(\R[\alpha])$. Theorem \ref{generator} gives a subset of $K$-linear generator for $\Tr(\R[\alpha])$, where its proof makes essentially use of Kang-Kashiwara's categorification result Proposition \ref{KK1}. Lemma \ref{relations} and Corollary \ref{zero element} give some useful relations inside the cocenter $\Tr(\R[\alpha])$. As an easy application of Theorem \ref{generator}, Proposition \ref{svv1a} proves the positivity of the degrees of the cocenter $\Tr(\R[\alpha])$. In Section 4 we introduce a new notion called ``piecewise dominant sequence'' in Definition \ref{pddfn}. In our first main result Theorem \ref{principle generator} we use this object to construct some explicit homogeneous elements which can span the maximal degree component as well as the minimal degree component of the cocenter of $\RR^\Lam_\alpha$. Moreover, a refined subset of $K$-linear generator for $\Tr(\R[\alpha])$ is also given in this theorem. As an application, Corollary \ref{maincor0} recovers a result of \cite[Theorem 3.31(a)]{SVV} on the degree range of the cocenter $\Tr(\R[\alpha])$. Furthermore, our second main results of this paper Theorems \ref{criterion} and \ref{maincor1} give some criteria for which $\R[\alpha]\neq 0$ (equivalently, $L(\Lam)_{\Lam-\alpha}\neq 0$), where $L(\Lam)$ is the irreducible highest weight module over $\mathfrak{g}$ of highest weight $\Lam$. More precisely, we show in Theorem \ref{criterion} that $\R[\alpha]\neq 0$ if and only if there exists a piecewise dominant sequence $\nu\in I^\alpha$, and construct in Theorem \ref{maincor1} an explicit (nonzero) monomial weight vector in $L(\Lam)_{\Lam-\alpha}$. Furthermore, we show in Lemma \ref{pathlem1} that we can associate each piecewise dominant sequence a crystal path in the crystal graph $\mathcal{B}(\Lam)$ of $L(\Lam)$, and show in Lemma \ref{pathlem2} each vertex in the crystal graph $\mathcal{B}(\Lam)$ can be connected with the highest weight vector $v_\Lam$ with a crystal path associated to some piecewise dominant sequence. We propose in Conjecture \ref{conj1} a refinement of the Indecomposability Conjecture. Our third main result of this paper (Theorem \ref{4mainthm}) verifies Conjecture \ref{conj1} in the case when $K$ is a field of {\it arbitrary} characteristic, $\{Q_{ij}(u,v)|i,j\in I\}$ are given as \cite[\S3.2.4]{Rou2}, $\mathfrak{g}$ is either symmetric and of finite type, or $\mathfrak{g}$ is of type $A_{\infty}$ or affine type $A_{(e-1)}^{(1)}$ with $e>1$.

\bigskip\bigskip
\centerline{Acknowledgements}
\bigskip

The research was supported by the National Natural Science Foundation of China (No. 12171029).
\bigskip

%%%%%%%%%%%%%%%%%%%%%%%%%%%%%%%%%%%%%%%%%%%%%%%%%%%%%%%%%%%%%%%%%
\section{Preliminary}

In this section we shall give some preliminary definitions and results on the quiver Hecke algebras and their cyclotomic quotients. Throughout, unless otherwise stated, we shall assume that $K$ is a field of arbitrary characteristic.

Let $I$ be an index set. An integral square matrix $A = (a_{i,j})_{i,j\in I}$ is called a \emph{symmetrizable generalized Cartan matrix} if it satisfies
\begin{enumerate}
  \item $a_{ii} = 2$, $\forall\, i \in I$;
  \item $a_{ij} \leqslant 0\,\, (i \neq j)$;
  \item $a_{ij} = 0 \Leftrightarrow \, a_{ji} = 0\,\, (i,j \in I)$;
  \item there is a diagonal matrix $D = \diag (d_i\in{\Z}_{>0}\mid i\in I)$ such that $DA$ is symmetric.
\end{enumerate}

A Cartan datum $(A,P,\Pi , P^{\vee} , \Pi^{\vee})$ consists of
\begin{enumerate}
  \item a symmetrizable generalized Cartan matrix $A$;
  \item a free abelian group $P$ of finite rank, called the \emph{weight lattice};
  \item $\Pi = \{ \alpha_i \in P \mid i\in I \}$, called the set of \emph{simple roots};
  \item $P^{\vee} := \Hom (P,{\Z})$, called the \emph{dual weight lattice} and $\<-,-\>: P^{\vee} \times P \to {\Z}$, the natural pairing;
  \item $\Pi^\vee = \{ h_i \mid i \in I \} \subset P^{\vee}$, called the set of \emph{simple coroots};
\end{enumerate}
satisfying the following properties:
\begin{enumerate}
  \item $\langle h_i , \alpha_j \rangle = a_{ij}$ for all $i,j\in I$,
  \item $\Pi$ is linearly independent,
  \item $\forall\, i \in I$, $\exists\,\Lambda_i \in P$ such that $\langle h_j, \Lambda_i \rangle = \delta_{ij}$ for all $j \in I$.
\end{enumerate}

Those $\Lambda_i$ are called the \emph{fundamental weights}. We set
$$P^+ := \{ \Lam \in P \mid \< h_i, \Lam \>\in {\Z}_{\geqslant 0}\text{ for all } i\in I \},$$
which is called the set of \emph{dominant integral weights}. The free abelian group $Q := \oplus_{i\in I}{\Z} \alpha_i$ is called the \emph{root lattice}.
Set $Q^+ = \sum_{i\in I} {\Z}_{\geqslant 0}\alpha_i$. For $\beta = \sum_{i\in I} k_i \alpha_i \in Q^+$, we define the \emph{height} of $\beta$ to be $|\beta| = \sum_{i\in I} k_i$. For each $n\in\mathbb{N}$, we set $$
Q_n^{+} := \{ \beta \in Q^+ \mid |\beta| = n\}. $$
Let $\mathfrak{g}=\mathfrak{g}(A)$ be the corresponding Kac-Moody Lie algebra associated to $A$ with Cartan subalgebra $\mathfrak{h}:= \Q \otimes_{{\Z}}P^\vee$. Since $A$ is symmetrizable, there is a symmetric bilinear form $(,)$ on $\mathfrak{h}^*$ satisfying
\begin{align*}
  & (\alpha_i, \alpha_j) = d_i a_{ij} \quad (i,j \in I) \quad \text{and} \\
  & \< h_i, \Lam \> = \frac{2(\alpha_i,\Lam)}{(\alpha_i,\alpha_i)} \quad \text{for any }\Lam \in \mathfrak{h}^* \text{ and } i\in I.
\end{align*}

%We can now recall the construction of Khovanov-Lauda-Rouquier algebra (or simply KLR algebra) $\RR_\beta$ associated with a Cartan datum $(A,P,\Pi,P^\vee ,\Pi^{\vee})$ and $\beta\in Q_n^+$.

Let $u,v$ be two commuting indeterminates over $K$. We fix a matrix $(Q_{i,j})_{i,j\in I}$ in $K[u,v]$ such that
\begin{align*}
  & Q_{i,j}(u,v) = Q_{j,i}(v,u), \quad  Q_{i,i}(u,v) = 0, \\
  & Q_{i,j}(u,v)=\sum_{p,q\geqslant 0} c_{i,j,p,q}\,u^pv^q, \,\,\,\ \text{if}\ i\neq j.
\end{align*}
where $c_{i,j,-a_{ij},0}\in K^\times$, and $c_{i,j,p,q}\neq 0$ only if $2(\alpha_i,\alpha_j)=-(\alpha_i,\alpha_i)p - (\alpha_j,\alpha_j)q$.

Let $\Sym_n = \<s_1 , \ldots , s_{n-1}\>$ be the symmetric group on $\{1,2,\cdots,n\}$, where $s_i = (i,i+1)$ is the transposition on $i,i+1$.
Then $\Sym_n$ acts naturally on $I^n$ by: $$w\nu := (\nu_{w^{-1}(1)},\ldots , \nu_{w^{-1}(n)}),$$
where $\nu = (\nu_1 , \ldots, \nu_n) \in I^n$.
The orbits of this action is identified with element of $Q_n^{+}$. Then $I^\beta:=\{\nu=(\nu_1,\cdots,\nu_n)\in I^n|\sum_{j=1}^{n}\alpha_{\nu_j}=\beta\}$ is the orbit corresponding to $\beta\in Q_n^+$.
%For any $k,m\in\mathbb{N}$ with $k\leq m$, we set $[k,m]:=\{k,k+1,\cdots,m\}$.
%For $\nu,\mu \in I^\beta$, define
%$$\SG (\mu,\nu) := \{ w\in \SG_n \mid w.\mu = \nu \}.$$

\begin{lem}[\text{\cite{KL1},\cite{KL2},\cite{Rou2}}] Let $\beta\in Q_n^+$. The elements in the following set form a $K$-basis of $\RR_\beta$: $$
\bigl\{x_1^{c_1}\cdots x_n^{c_n}\psi_w e(\nu)\bigm|\nu\in I^\beta, w\in\Sym_n, \mathbf{c}_1,\cdots,c_n\in\N\bigr\} .
$$
\end{lem}

\begin{dfn}\label{KLR}
 Let $\beta\in Q_n^+$. The quiver Hecke algebra $\RR_\beta$ associated with a Cartan datum $(A,P,\Pi,P^{\vee},\Pi^{\vee})$, $(Q_{i,j})_{i,j\in I}$ and $\beta \in Q_n^+$ is the associative algebra over $K$ generated by $e(\nu)\, (\nu\in I^\beta)$, $x_k\, (1\leqslant k \leqslant n)$, $\tau_l \, (1\leqslant l \leqslant n-1)$ satisfying the following defining relations:
  \begin{equation*}
    \begin{aligned}
      & e(\nu) e(\nu') = \delta_{\nu, \nu'} e(\nu), \ \
      \sum_{\nu \in I^{\beta}}  e(\nu) = 1, \\
      & x_{k} x_{l} = x_{l} x_{k}, \ \ x_{k} e(\nu) = e(\nu) x_{k}, \\
      & \tau_{l} e(\nu) = e(s_{l}(\nu)) \tau_{l}, \ \ \tau_{k} \tau_{l} = \tau_{l} \tau_{k} \ \ \text{if} \ |k-l|>1, \\
      & \tau_{k}^2 e(\nu) = Q_{\nu_{k}, \nu_{k+1}} (x_{k}, x_{k+1})e(\nu), \\
      & (\tau_{k} x_{l} - x_{s_k(l)} \tau_{k}) e(\nu) = \begin{cases}
      -e(\nu) \ \ & \text{if} \ l=k, \nu_{k} = \nu_{k+1}, \\
      e(\nu) \ \ & \text{if} \ l=k+1, \nu_{k}=\nu_{k+1}, \\
      0 \ \ & \text{otherwise},
      \end{cases} \\[.5ex]
      & (\tau_{k+1} \tau_{k} \tau_{k+1}-\tau_{k} \tau_{k+1} \tau_{k}) e(\nu)\\
      &\hspace*{8ex} =\begin{cases} \dfrac{Q_{\nu_{k}, \nu_{k+1}}(x_{k},
      x_{k+1}) - Q_{\nu_{k+2}, \nu_{k+1}}(x_{k+2}, x_{k+1})}{x_{k} - x_{k+2}}e(\nu) \ \ & \text{if} \
      \nu_{k} = \nu_{k+2}, \\
      0 \ \ & \text{otherwise}.
      \end{cases}
    \end{aligned}
  \end{equation*}
\end{dfn}

In particular, $\RR_0 \cong K$, and $\RR_{\alpha_i}$ is isomorphic to $K[x_1]$.
For any $\beta\in Q_n^+$ and $i\in I$, we set $$
e(\beta, i) = \operatornamewithlimits{\sum}\limits_{\nu=(\nu_1,\cdots,\nu_n)\in I^{\beta}}e(\nu_1,\cdots,\nu_n,i)\in\RR_{\beta+\alpha_i}. $$
% and more generally, for $\beta\geq\beta'\in Q_{n-t+1}^+$, and $\nu' \in I^{\beta'}$, $$e(\beta -\beta' , \nu'): = \operatornamewithlimits{\sum}\limits_{\substack{\nu\in I^\beta \\ (\nu_t , \ldots ,\nu_n) = \nu'}}e(\nu)$$  We also abbreviate $e(\underbrace{i\cdots i}_{k\,\text{copies}})$ as $e(i^k)$.

The algebra $\RR_\beta$ is ${\Z}$-graded whose grading is given by
\begin{equation*}
  \deg e(\nu) = 0, \qquad \deg x_k e(\nu) = (\alpha_{\nu_k},\alpha_{\nu_k}), \qquad \deg \tau_l e(\nu) = - (\alpha_{\nu_l},\alpha_{\nu_{l+1}}).
\end{equation*}

Let $\Lambda \in P^+$ be a dominant integral weight. We now recall the cyclotomic quiver Hecke algebra $\RR_\beta^\Lambda$.
For $1\leqslant k \leqslant n$, we define
\begin{equation*}
  a^\Lambda_\beta (x_k) = \sum_{\nu\in I^\beta} x_k^{\<h_{\nu_k},\Lam\>}e(\nu) \in \RR_\beta.
\end{equation*}

\begin{dfn}\label{cyclotomicKLR}
  Set $I_{\Lambda,\beta} = \RR_\beta a^\Lambda_\beta (x_1) \RR_\beta$. The {cyclotomic quiver Hecke algebra} $\RR^\Lambda_\beta$ is defined to be the quotient algebra:   $$
  \RR_\beta^\Lambda = \RR_\beta/I_{\Lambda,\beta}. $$
\end{dfn}
Sometimes we shall write $\R[\beta](K)$ instead of $\R[\beta]$ in order to emphasize the ground field $K$. In general, if $\O$ is a commutative ring and
$Q_{ij}(u,v)\in\O[u,v]$ for any $i,j\in I$, then we can define the cyclotomic quiver Hecke algebra $\R[\beta](\O)$ over $\O$.

For any $K$-algebra $A$, we define the center $Z(A):=\{a\in A|ax=xa,\forall\,x\in A\}$, and define the cocenter $\Tr(A)$ of $A$ to be the $K$-linear space $\Tr(A):=A/[A,A]$, where $[A,A]$ is the $K$-subspace of $A$ generated by all commutators of the form $xy-yx$ for $x,y\in A$. Note that $\Tr(A)$ is the $0$-th Hochschild homology ${\rm H\!H}_0(A)$ of $A$, while $Z(A)$ is the $0$-th Hochschild cohomology ${\rm H\!H}^0(A)$ of $A$.

Let $B$ be a $K$-algebra with an algebra homomorphism $i: B\rightarrow A$. Then $A$ naturally becomes a $(B,B)$-bimodule. For any $f\in\{a\in A|ab=ba,\forall\,b\in B\}$, we define (following \cite[(1)]{SVV}) $$
\mu_f: A\otimes_{B}A\rightarrow A,\quad \sum_{(a)}a_1\otimes a_2\rightarrow \sum_{(a)}a_1fa_2 .
$$

\begin{dfn} Let $\Lam\in P^+$, $\beta\in Q_n^+$ and $\nu=(\nu_1,\cdots,\nu_n)\in I^\beta$. We define $$
d_{\Lambda ,\beta}:=2(\Lam,\beta)-(\beta,\beta).
$$
For any $1\leq k\leq n$, we set $$
 \lambda_{k-1,\nu_k}=\<h_{\nu_k},\Lambda-\sum_{j=1}^{k-1}\alpha_{\nu_k}\>.
$$
\end{dfn}

Let $\alpha\in Q_n^+$, $z\in\RR_{\alpha}$ and $m>n$. By convention, we shall often abbreviate the element $$z\sum_{i_{n+1},\cdots,i_m\in I}e(\alpha,i_{n+1},\cdots,i_m)\in
\RR_{m}:=\oplus_{\beta\in Q_m^+}\RR_\beta$$ as $z$. The same convention is also adopted for elements in $\R[\alpha]$. The following result of Kang-Kashiwara will be used in  the proof of the main results in this paper.

\begin{lem}[\text{\rm \cite[Theorem 5.2]{KK}, \cite[(6),(7)]{SVV}}]\label{KK1} Let $\alpha\in Q_n^+$, $i\in I$ and $z\in e(\alpha,i)\R[\alpha,i]e(\alpha,i)$. \begin{enumerate}
\item[1)] If $\lam_{n,i}\geq 0$, then there are unique elements $\pi(z)\in \R[\alpha]e(\alpha-\alpha_i,i)\otimes_{\R[\alpha-\alpha_i]}e(\alpha-\alpha_i,i)\R[\alpha]$ and
$p_k(z)\in\R[\alpha]$ such that\footnote{We remark that there is a typo in \cite[(6)]{SVV}, where the element ``$e(\alpha,i)$'' was missing in the second term of the righthand side of (\ref{geq0}).} \begin{equation}\label{geq0}
z=\mu_{\tau_n}(\pi(z))+\sum_{k=0}^{\lam_{n,i}-1}p_k(z)x_{n+1}^ke(\alpha,i);
\end{equation}
\item[2)] If $\lam_{n,i}\leq 0$, then there is a unique element $\widetilde{z}\in \R[\alpha]e(\alpha-\alpha_i,i)\otimes_{\R[\alpha-\alpha_i]}e(\alpha-\alpha_i,i)\R[\alpha]$ and
$p_k(z)\in\R[\alpha]$ such that  \begin{equation}\label{leq0}
z=\mu_{\tau_n}(\widetilde{z}),\quad \mu_{x_n^k}(\widetilde{z})=0,\,\,\forall\,0\leq k\leq -\lam_{n,i}-1.
\end{equation}
\end{enumerate}
\end{lem}

Let $\alpha\in Q_n^+$, $\nu\in I^\alpha$ and $i\in I$. Following \cite[Theorem 3.8]{SVV}, we define $$\begin{aligned}
\hat{\eps}'_{i,\Lam-\alpha}:\quad e(\alpha,i)\R[\alpha+\alpha_i]e(\alpha,i)&\rightarrow \R[\alpha]\\
z&\mapsto\begin{cases} p_{\lam_{n,i}-1}(z), &\text{if $\lam_{n,i}>0$;}\\
\mu_{x_n^{-\lam_{n,i}}}(\widetilde{z}), &\text{if $\lam_{n,i}\leq 0$.}
\end{cases}
\end{aligned}
$$
Let $r_\nu\in K^\times$ be defined as \cite[(62)]{SVV}. For any $\nu,\nu'\in I^\alpha$, $z\in\R[\alpha]$, we define $$
t_{\Lam,\alpha}(e(\nu)ze(\nu')):=\begin{cases} 0, &\text{if $\nu\neq \nu'$;}\\
r_{\nu}\hat{\eps}_{n,\nu_n}\hat{\eps}_{n-1,\nu_{n-1}}\cdots \hat{\eps}_{1,\nu_1}\bigl(e(\nu)ze(\nu')\bigr), &\text{if $\nu=\nu'$,}\end{cases}
$$
where $\hat{\eps}_{k,\nu_k}$ is a map $$
\hat{\eps}_{k,\nu_k}: e(\nu_1,\cdots,\nu_k)\R[\alpha-\sum_{j=k+1}^{n}\alpha_{\nu_j}]e(\nu_1,\cdots,\nu_k)\rightarrow e(\nu_1,\cdots,\nu_{k-1})\R[\alpha-\sum_{j=k}^{n}\alpha_{\nu_j}]e(\nu_1,\cdots,\nu_{k-1}),
$$
which is the restriction of $\hat{\eps}'_{\nu_k,\Lam-\sum_{j=1}^{k-1}\alpha_{\nu_j}}$. Note that the map $\hat{\eps}_{k,\nu_k}$ was denote by $\hat{\eps}_{\nu_k}$ in \cite[A.3]{SVV}. We extend $t_{\Lam,\alpha}$ linearly to a $K$-linear map $t_{\Lam,\alpha}: \R[\alpha]\rightarrow K$.

\begin{lem}[\text{\rm \cite[Proposition 3.10]{SVV}}]\label{tLam} The map $t_{\Lam,\alpha}: \R[\alpha]\rightarrow K$ is a homogeneous symmetrizing form on $\R[\alpha]$ of degree $-d_{\Lam,\alpha}$.
\end{lem}

In \cite{HS}, we have obtained some closed formulae for the graded dimension of the cyclotomic quiver Hecke algebra $\R[\beta]$ of arbitrary type.

\begin{lem}[\text{\rm \cite[Theorem 1.1]{HS}}]\label{HS1}  Let $\beta\in Q_n^+$ and $\nu=(\nu_1,\cdots,\nu_n),\nu'=(\nu'_1,\cdots,\nu'_n)\in I^\beta$. Then $$
\dim_q e(\nu)\RR^\Lam_{\beta}e(\nu')=\sum_{\substack{w\in\Sym(\nu,\nu')}}\prod_{t=1}^{n}\Bigl([N^{\Lam}(w,\nu,t)]_{\nu_t}
q_{\nu_t}^{N^{\Lam}(1,\nu,t)-1}\Bigr).
$$
where $\Sym(\nu,\nu'):=\{w\in\Sym_n|w\nu=\nu'\}$, $q_{\nu_t}:=q^{d_{\nu_t}}$, $[m]_{\nu_t}$ is the quantum integer (\cite[(2.1)]{HS}), $N^{\Lam}(w,\nu,t)$ is defined as follows: $$
N^\Lam(w,\nu,t):=\<h_{\nu_t}, \Lam-\sum_{j\in J_w^{<t}}\alpha_{\nu_j}\>,\quad
J_w^{<t}:=\{1\leq j<t|w(j)<w(t)\}.
$$
\end{lem}

The above lemma shows that the dimension of $\R[\alpha](K)$ depends only on the root system associated to $A$ and the dominant weight $\Lam$, but not on the chosen ground field $K$ and the polynomials $Q_{ij}(u,v)$. This implies that if each $Q_{ij}(u,v)$ is defined over $\Z$ then $\R[\beta](\Z)$ is free over $\Z$, and hence $\O\otimes_{\Z}\R[\beta](\Z)\cong\R[\beta](\O)$ for any commutative ground ring $\O$. Thus we recover the following result of Ariki, Park and Speyer.

\begin{cor}[\text{\rm \cite[Proposition 2.4]{APS}}]\label{free} Suppose that each $Q_{ij}(u,v)$ is defined over $\Z$. For any commutative ground ring $\O$, the cyclotomic quiver Hecke algebra $\R[\beta](\O)$ is a free $\O$-module of finite rank.
\end{cor}

For any $A_1,\cdots,A_p\in\R[\alpha]$, we define the ordered product: $$
\overrightarrow{\prod\limits_{1\leq i\leq p}}A_i:=A_1A_2\cdots A_p .
$$

\bigskip

\section{Relations and $K$-linear generators of the cocenter}

In this section we shall investigate some relations inside the cocenter $\Tr(\R[\alpha])$ with a purpose of looking for some normal forms for the $K$-linear generators of the cocenter $\Tr(\R[\alpha])$. We shall also analyze the range of the degrees of elements in $\Tr(\R[\alpha])$. The main results of this section is Theorem \ref{generator}.

\subsection{$K$-linear generators}

Let $\alpha\in Q_n^+$. In this subsection, we will give a set of $K$-linear generators for the cocenter of the cyclotomic quiver Hecke algebra $\R[\alpha]$.

Recall that a composition of $n$ is a sequence of non-negative integers $\mathbf{a}=(a_1,a_2,\cdots,a_k)$ such that $\sum_{i=1}^{k}a_i=n$. If $\mathbf{a}=(a_1,a_2,\cdots,a_k)$ is a composition of $n$ then we write $\mathbf{a}\models n$.

Let $\nu=(\nu_1,\cdots,\nu_n)\in I^\alpha$. We define $$
\mathcal{C}(\nu):=\Biggl\{\mathbf{b}=(b_1,\cdots,b_p)\models n\Biggm|\begin{matrix}\text{$p,b_1,\cdots,b_p\geq 1$, $\forall\,1\leq i\leq p$ and}\\
\text{$\sum_{k=1}^{i-1}b_k+1\leq j<\sum_{k=1}^{i}b_k$, $\nu_j=\nu_{j+1}$.}
\end{matrix}\Biggr\}
$$

For any $\mathbf{b}=(b_1,\cdots,b_p)\in \mathcal{C}(\nu)$, we define $\bc:=(c_0,c_1,\cdots,c_{p-1},c_p)$, where \begin{equation}\label{cj1}
c_0:=0,\quad c_j:=b_1+b_2+\cdots+b_j,\,\,j=1,2,\cdots,p .
\end{equation}

Therefore, for any $\mathbf{b}=(b_1,\cdots,b_p)\in \mathcal{C}(\nu)$, we can decompose $\nu$ as follows:
\begin{equation}\label{nu1}
\nu=(\underbrace{\nu^{1},\nu^{1},\cdots,\nu^{1}}_{b_{1}\,copies},\cdots,\underbrace{\nu^{p},\nu^{p},\cdots,\nu^{p}}_{b_{p}\,copies}),
\end{equation}
where $p,b_1,\cdots,b_p\in\Z^{\geq 1}$ with $\sum_{i=1}^{p}b_i=n$, $\nu^1,\cdots,\nu^p\in I$. Note that it could happen that $\nu^{j}=\nu^{j+1}$ for some $1\leq j<p$. We define \begin{equation}\label{cLamnu}
\mathcal{C}^\Lam(\nu):=\biggl\{\mathbf{b}=(b_1,\cdots,b_p)\in\mathcal{C}(\nu)\biggm|\lam_{c_i,\nu^{i+1}}>0, \forall\,0\leq i<p\biggr\}.
\end{equation}

\begin{dfn} Let $\nu\in I^\alpha$. We denote by $\R[\nu,1]$ the $K$-subspace of $\R[\alpha]$ spanned by the elements of the following form: \begin{equation}
\Biggl\{\overrightarrow{\prod\limits_{0\leq i< p}}\Bigl(x_{c_i+1}^{n_{\bc,i+1}}\tau_{c_i+1}\tau_{c_i+2}\cdots\tau_{c_{i+1}-1}\Bigr)e(\nu)\Biggm|\begin{matrix}\text{$0\leq n_{\bc,i+1}<\lam_{c_i,\nu^{i+1}}, \forall\,0\leq i<p$, where}\\
\text{$\mathbf{b}\in\mathcal{C}^\Lam(\nu)$, $\{c_j|1\leq j\leq p\}$ is as in (\ref{cj1}),}\\ \text{$\{\nu^i|1\leq i\leq p\}$ is as in (\ref{nu1}).} \end{matrix}\Biggr\}.
\end{equation}
\end{dfn}

For any subset $A$ of $\R[\alpha]$, we use $\overline{A}$ to denote the natural image of $A$ in the cocenter $\R[\alpha]/[\R[\alpha],\R[\alpha]]$ of $\R[\alpha]$.
%The following theorem is the first main result of this paper:

\begin{thm}\label{generator} We have $$
\Tr(\R[\alpha])=\RR^\Lam_\alpha/[\RR^\Lam_\alpha,\RR^\Lam_\alpha]=\sum_{\nu\in I^\alpha}\overline{\RR^\Lam_{\nu,1}}.$$
\end{thm}

\begin{proof} We use induction on $n$. The case $n=1$ is trivial. Now we suppose our Theorem holds for $n-1\geq 1$. Let $z\in e(\nu)\R[\alpha]e(\mu)$. If $\nu\neq\mu$, then in the cocenter we have $$\overline{ze(\mu)}=\overline{e(\mu)z}=\overline{0} .
$$
Hence we only need to consider the case when $\nu=\mu$. Henceforth we assume $\nu_n=i$.

\medskip
{\it Case 1.} Suppose $\lambda_{n-1,i}>0$. Then by Lemma \ref{KK1} there are unique elements $\pi(z)\in \RR^\Lam_{\alpha-\alpha_i}e(\alpha-2\alpha_i,i)\otimes_{\RR^\Lam_{\alpha-2\alpha_i}}e(\alpha-2\alpha_i,i)\RR^\Lam_{\alpha-\alpha_i}$ and $p_k(z)\in \RR^\Lam_{\alpha-\alpha_i}$ such that \begin{equation}\label{zdecomp1} z=\mu_{\tau_{n-1}}(\pi(z))e(\alpha-\alpha_i,i)+\sum_{k=0}^{\lam_{n-1,i}-1}p_k(z)x_n^k e(\alpha-\alpha_i,i) .
\end{equation}
Since $x_n$ centralize the image of $\RR^\Lam_{\alpha-\alpha_i}$ in $e(\alpha-\alpha_i,i)\RR^\Lam_{\alpha}e(\alpha-\alpha_i,i)$, we can apply induction hypothesis to $p_k(z)$ to deduce that $p_k(z)x_n^k\in \sum_{\rho\in I^\alpha}\RR^\Lam_{\rho,1}+[\R[\alpha],\R[\alpha]]$ for each $0\leq k\leq \lam_{n-1,i}-1$. It remains to consider the first term in the righthand side of (\ref{zdecomp1}). For the first term in (\ref{zdecomp1}), we can write $$\pi(z)=\sum_{l=1}^{m_1}z_{l1}\otimes z_{l2},$$ where $z_{l1}\in\RR^\Lam_{\alpha-\alpha_i}e(\alpha-2\alpha_i,i)$ and $z_{l2}\in e(\alpha-2\alpha_i,i)\RR^\Lam_{\alpha-\alpha_i}$. Hence, we have $$\begin{aligned}\mu_{\tau_{n-1}}(\pi(z))e(\alpha-\alpha_i,i)&=\sum_{l=1}^{m_1}z_{l1}e(\alpha-2\alpha_i,i,i)\tau_{n-1}e(\alpha-2\alpha_i,i,i)z_{l2}\\
&\equiv\sum_{l=1}^{m_1}e(\alpha-2\alpha_i,i,i)z_{l2}z_{l1}e(\alpha-2\alpha_i,i,i)\tau_{n-1}\\
&\equiv e(\alpha-2\alpha_i,i,i)he(\alpha-2\alpha_i,i,i)\tau_{n-1}\pmod{[\R[\alpha],\R[\alpha]]},
\end{aligned}$$ where $h\in e(\alpha-2\alpha_i,i)\RR^\Lam_{\alpha-\alpha_i}e(\alpha-2\alpha_i,i)$.

Consider $e(\alpha-2\alpha_i,i)$. Since $\lam_{n-2,i}=\lambda_{n-1,i}+2\geq 3>0$, we can use Lemma \ref{KK1} again to find unique elements $\pi(h)\in \RR^\Lam_{\alpha-2\alpha_i}e(\alpha-3\alpha_i,i)\otimes_{\RR^\Lam_{\alpha-3\alpha_i}}e(\alpha-3\alpha_i,i)\RR^\Lam_{\alpha-2\alpha_i}$ and $p_k(h)\in \RR^\Lam_{\alpha-2\alpha_i}$ for each $0\leq k\leq\lam_{n-1,i}+1$, such that \begin{equation}\label{zdecomp2}
h=\mu_{\tau_{n-2}}(\pi(h))e(\alpha-2\alpha_i,i)+\sum_{k=0}^{\lam_{n-1,i}+1}p_k(h)x_{n-1}^ke(\alpha-2\alpha_i,i),\end{equation}
Using the same argument as before (i.e., induction hypothesis), we can deduce that $$
p_k(h)x_{n-1}^k\tau_{n-1}e(\alpha-2\alpha_i,i,i)\in \sum_{\rho\in I^\alpha}\RR^\Lam_{\rho,1}+[\R[\alpha],\R[\alpha]],\,\,\,\forall\,\,0\leq k\leq \lam_{n-1,i}+1.
$$
Hence, we only need to consider $\mu_{\tau_{n-2}}(\pi(h))\tau_{n-1}e(\alpha-2\alpha_i,i,i)$. As the same computation above, we write  $$\pi(h)=\sum_{l=1}^{m_2}h_{l1}\otimes h_{l2},$$ where $h_{l1}\in\RR^\Lam_{\alpha-2\alpha_i}e(\alpha-3\alpha_i,i)$ and $h_{l2}\in e(\alpha-3\alpha_i,i)\RR^\Lam_{\alpha-2\alpha_i}$.  Also note that $\tau_{n-1}$ centralize the image of $\RR^\Lam_{\alpha-2\alpha_i}e(\alpha-3\alpha_i,i)$ and $e(\alpha-3\alpha_i,i)\RR^\Lam_{\alpha-2\alpha_i}$ in $e(\alpha-2\alpha_i,i,i)\RR^\Lam_{\alpha}e(\alpha-2\alpha_i,i,i)$, we can deduce $$
\begin{aligned}
&\quad\,\mu_{\tau_{n-2}}(h)\tau_{n-1}e(\alpha-3\alpha_i,i)\\
&=\sum_{l=1}^{m_2}h_{l1}e(\alpha-3\alpha_i,i,i,i)\tau_{n-2}e(\alpha-3\alpha_i,i,i,i)h_{l2}\tau_{n-1}\\
&=\sum_{l=1}^{m_2}h_{l1}e(\alpha-3\alpha_i,i,i,i)\tau_{n-2}\tau_{n-1}e(\alpha-3\alpha_i,i,i,i)h_{l2}\\
&\equiv\sum_{l=1}^{m_2}e(\alpha-3\alpha_i,i,i,i)h_{l2}h_{l1}e(\alpha-3\alpha_i,i,i,i)\tau_{n-2}\tau_{n-1}\\
&\equiv e(\alpha-3\alpha_i,i,i,i)h'e(\alpha-3\alpha_i,i,i,i)\tau_{n-2}\tau_{n-1}\pmod{[\R[\alpha],\R[\alpha]]},
\end{aligned}$$
where $h'\in e(\alpha-3\alpha_i,i)\RR^\Lam_{\alpha-2\alpha_i}e(\alpha-3\alpha_i,i)$. It remains to show that $$
e(\alpha-3\alpha_i,i,i,i)h'e(\alpha-3\alpha_i,i,i,i)\tau_{n-2}\tau_{n-1}\in \sum_{\rho\in I^\alpha}\RR^\Lam_{\rho,1}+[\R[\alpha],\R[\alpha]].
$$
Next consider $e(\alpha-3\alpha_i,i)$ and replace $h$ with $h'$ in (\ref{zdecomp2}). Repeating the previous argument and noting that the number of $i$ occurred in $\alpha$ is finite, we will end up with an element of the form: $$
\hat{z}=\hat{h}\tau_{n-s+1}\cdots\tau_{n-1}e(\alpha-s\alpha_i,\underbrace{i,i,\cdots,i}_{s\,copies}),$$ where $\hat{h}\in e(\alpha-s\alpha_i,i)\RR^\Lam_{\alpha-(s-1)\alpha_i}e(\alpha-s\alpha_i,i)$ and $$\mu_{\tau_{n-s}}(\pi(\hat{h}))=0,\quad
\hat{h}=\sum_{t=0}^{\lam_{n-1,i}+2s-1} p_k(\hat{h})x_{n-s+1}^t, $$
where $p_k(\hat{h})\in \RR^\Lam_{\alpha-s\alpha_i}$. But $x_{n-s+1}^t\tau_{n-s+1}\cdots\tau_{n-1}$ centralize the image of $\RR^\Lam_{\alpha-s\alpha_i}$ in $e(\alpha-s\alpha_i,\underbrace{i,i,\cdots,i}_{s\,copies})\RR^\Lam_{\alpha}e(\alpha-s\alpha_i,\underbrace{i,i,\cdots,i}_{s\,copies})$. The induction hypothesis implies that for any $0\leq t\leq \lam_{n-1,i}+2s-1$,  $$p_k(\hat{h})x_{n-s+1}^t\tau_{n-s+1}\cdots\tau_{n-1}e(\alpha-s\alpha_i,\underbrace{i,i,\cdots,i}_{s\,copies})\in  \sum_{\rho\in I^\alpha}\RR^\Lam_{\rho,1}+[\R[\alpha],\R[\alpha]].$$
This proves $z\in \sum_{\rho\in I^\alpha}\RR^\Lam_{\rho,1}+[\R[\alpha],\R[\alpha]]$ in the case $\lam_{n-1,i}\geq 0$.

\medskip
{\it Case 2.} Suppose $\lam_{n-1,i}\leq 0$. Then by Lemma \ref{KK1} there is an unique element $\widetilde{z}\in \RR^\Lam_{\alpha-\alpha_i}e(\alpha-2\alpha_i,i)\otimes_{\RR^\Lam_{\alpha-2\alpha_i}}e(\alpha-2\alpha_i,i)\RR^\Lam_{\alpha-\alpha_i}$ such that $$
\mu_{\tau_{n-1}}(\widetilde{z})e(\alpha-\alpha_i,i)=z,\quad \mu_{x_n^k}(\widetilde{z})e(\alpha-\alpha_i,i)=0,\,\,\forall\,0\leq k\leq -\lam_{n-1,i}-1. $$
we can write $$\widetilde{z}=\sum_{l=1}^{\breve{m}_1}z_{l1}\otimes z_{l2},$$ where $z_{l1}\in\RR^\Lam_{\alpha-\alpha_i}e(\alpha-2\alpha_i,i)$ and $z_{l2}\in e(\alpha-2\alpha_i,i)\RR^\Lam_{\alpha-\alpha_i}$. Hence, we have $$\begin{aligned}\mu_{\tau_{n-1}}(\widetilde{z})e(\alpha-\alpha_i,i)&=\sum_{l=1}^{\breve{m}_1}z_{l1}e(\alpha-2\alpha_i,i,i)\tau_{n-1}e(\alpha-2\alpha_i,i,i)z_{l2}\\
&\equiv\sum_{l=1}^{\breve{m}_1}e(\alpha-2\alpha_i,i,i)z_{l2}z_{l1}e(\alpha-2\alpha_i,i,i)\tau_{n-1}\\
&\equiv e(\alpha-2\alpha_i,i,i)h(1)e(\alpha-2\alpha_i,i,i)\tau_{n-1}\pmod{[\R[\alpha],\R[\alpha]]},
\end{aligned}$$ where $h(1)\in e(\alpha-2\alpha_i,i)\RR^\Lam_{\alpha-\alpha_i}e(\alpha-2\alpha_i,i)$. It remains to show that $e(\alpha-2\alpha_i,i,i)h(1)e(\alpha-2\alpha_i,i,i)\tau_{n-1}\in \sum_{\rho\in I^\alpha}\RR^\Lam_{\rho,1}+[\R[\alpha],\R[\alpha]]$.

Now consider $e(\alpha-2\alpha_i,i)$. Note that $\lam_{n-2,i}=\lambda_{n-1,i}+2$. If $\lambda_{n-1,i}+2\leq 0$, then repeating the above argument (replacing $\widetilde{z}$ with $\widetilde{h(1)}$) we can get an element $h(2)\in e(\alpha-3\alpha_i,i)\RR^\Lam_{\alpha-2\alpha_i}e(\alpha-3\alpha_i,i)$, and it remains to show that  \begin{equation}\label{h21}e(\alpha-3\alpha_i,i,i,i)h(2)e(\alpha-3\alpha_i,i,i,i)\tau_{n-2}\tau_{n-1}\in \sum_{\rho\in I^\alpha}\RR^\Lam_{\rho,1}+[\R[\alpha],\R[\alpha]]. \end{equation}

If $\lambda_{n-1,i}+2>0$, then we can use Lemma \ref{KK1} again to find unique elements $\pi(h(1))\in \RR^\Lam_{\alpha-2\alpha_i}e(\alpha-3\alpha_i,i)\otimes_{\RR^\Lam_{\alpha-3\alpha_i}}e(\alpha-3\alpha_i,i)\RR^\Lam_{\alpha-2\alpha_i}$ and $p_k(h(1))\in \RR^\Lam_{\alpha-2\alpha_i}$ for each $0\leq k\leq\lam_{n-1,i}+1$, such that \begin{equation}\label{zdecomp3}
h(1)=\mu_{\tau_{n-2}}(\pi(h(1)))e(\alpha-2\alpha_i,i)+\sum_{k=0}^{\lam_{n-1,i}+1}p_k(h(1))x_{n-1}^ke(\alpha-2\alpha_i,i),\end{equation}
Using the same argument as before (i.e., induction hypothesis), we can deduce that $$
p_k(h(1))x_{n-1}^k\tau_{n-1}e(\alpha-2\alpha_i,i,i)\in \sum_{\rho\in I^\alpha}\RR^\Lam_{\rho,1}+[\R[\alpha],\R[\alpha]],\,\,\,\forall\,\,0\leq k\leq \lam_{n-1,i}+1.
$$
Hence, we only need to consider $\mu_{\tau_{n-2}}(\pi(h(1)))\tau_{n-1}e(\alpha-2\alpha_i,i,i)$. We can write $$\pi(h(1))=\sum_{l=1}^{\breve{m}_2}h_{l1}\otimes h_{l2},$$ where $h_{l1}\in\RR^\Lam_{\alpha-2\alpha_i}e(\alpha-3\alpha_i,i)$ and $h_{l2}\in e(\alpha-3\alpha_i,i)\RR^\Lam_{\alpha-2\alpha_i}$. Hence, we have $$\begin{aligned}&\quad\,\mu_{\tau_{n-2}}(\pi(h(1)))\tau_{n-1}e(\alpha-2\alpha_i,i,i)\\
&=\sum_{l=1}^{\breve{m}_2}h_{l1}e(\alpha-3\alpha_i,i,i,i)\tau_{n-2}e(\alpha-3\alpha_i,i,i,i)h_{l2}\tau_{n-1}\\
&\equiv\sum_{l=1}^{\breve{m}_2}e(\alpha-3\alpha_i,i,i,i)h_{l2}h_{l1}e(\alpha-3\alpha_i,i,i,i)\tau_{n-2}\tau_{n-1}\\
&\equiv e(\alpha-3\alpha_i,i,i,i)h(2)e(\alpha-3\alpha_i,i,i,i)\tau_{n-2}\tau_{n-1}\pmod{[\R[\alpha],\R[\alpha]]},
\end{aligned}$$ where $h(2)\in e(\alpha-3\alpha_i,i)\RR^\Lam_{\alpha-2\alpha_i}e(\alpha-3\alpha_i,i)$. It remains to show that \begin{equation}\label{h22}
e(\alpha-3\alpha_i,i,i,i)h(2)e(\alpha-3\alpha_i,i,i,i)\tau_{n-2}\tau_{n-1}\in \sum_{\rho\in I^\alpha}\RR^\Lam_{\rho,1}+[\R[\alpha],\R[\alpha]].
\end{equation}

Note that (\ref{h22}) is formally the same as (\ref{h21}). Therefore, we are in a position to repeat the same argument as before. As we discussed in the last paragraph of Case 1, this procedure will end after a finite number of steps. As a result, we can deduce that $z\in \sum_{\rho\in I^\alpha}\RR^\Lam_{\rho,1}+[\R[\alpha],\R[\alpha]]$. This completes the proof in Case 2 and hence the proof of the whole theorem.
\end{proof}

%\begin{cor} We have $$
%\Tr(\R[\alpha])=\text{$K$-{\rm Span}}\Biggl\{\overline{\psi_{w}}\prod_{i=1}^{p}\overline{x_{c_{i-1}+1}^{k_i}}\Biggm|\begin{matrix}\text{$\nu\in I^\alpha,\mathbf{b}=(b_1,\cdots,b_p)\in\mathcal{C}(\nu)$,}\\
%\text{$w\in\Sym_{\mathbf{b}}$, $0\leq k_{c_{i-1}+1}<\lam_{c_i-1,\nu^i},\forall\,1\leq i\leq p$,}\end{matrix}\Biggr\},
%$$
%where $\{c_j|0\leq j\leq p\}$ is as defined in (\ref{cj1}).
%\end{cor}
\smallskip

\subsection{Positivity of the degree of the cocenter}

Let $\alpha\in Q_n^+$. The purpose of this subsection is to give an application of Theorem \ref{generator}. We shall show that any element in the cocenter $\Tr(\R[\alpha])$ of $\R[\alpha]$ has degree $\geq 0$. Equivalently, this means any element in the center $Z(\R[\alpha])$ of $\R[\alpha]$ has degree $\leq d_{\Lam,\alpha}$.

Let $\ell,n\in\N$. Consider the cyclotomic nilHecke algebra $\NH_n^\ell$ of type $A$ which is defined over $K$. Applying \cite[Theorem 3.7]{HuL}, we see that the center $Z(\NH_n^\ell)$ of the cyclotomic nilHecke algebra $\NH_n^\ell$ has a basis $\{z_\mu|\mu\in\mathscr{P}_0\}$, where $\mathscr{P}_0$ is defined in the paragraph above \cite[Definition 2.5]{HuL}. The degree of each basis element $z_\mu$ is explicitly known by \cite[Definition 3.3]{HuL}. In particular, we know that \begin{equation}\label{degnilhecke}
\text{$\bigl(Z(\NH_n^\ell)\bigr)_j\neq 0$\,\,\, only if\,\,\, $0\leq j\leq 2\ell n-2n^2$,}
\end{equation}
and $\dim\bigl(Z(\NH_n^\ell)\bigr)_0=\dim\bigl(Z(\NH_n^\ell)\bigr)_{2\ell n-2n^2}=1$. As a result, we can deduce that \begin{equation}\label{degnilhecke2}
\text{$\Tr(\NH_n^\ell)_j\neq 0$\,\,\, only if\,\,\, $0\leq j\leq 2\ell n-2n^2$,}
\end{equation}
and $\dim\Tr(\NH_n^\ell)_0=\dim\Tr(\NH_n^\ell)_{2\ell n-2n^2}=1$.

By \cite[Corollary 4.4]{KK}, $\R[\alpha]$ is a finite dimensional $K$-linear space. Let $\nu\in I^\alpha$. For each $1\leq k\leq n$, we have that
$\deg x_ke(\nu)>0$. It follows that $x_ke(\nu)$ is nilpotent in $\R[\alpha]$.

Now let $\mathbf{b}=(b_1,\cdots,b_p)\in\mathcal{C}(\nu)$. We define $\bc=(c_0,c_1,\cdots,c_{p})$ as in (\ref{cj1}). For each $0\leq j\leq p-1$, we use $l_j$ to denote the nilpotent index of $x_{c_j+1}e(\nu)$. That says, $$
l_j:=\min\biggl\{l\geq 1\biggm|\bigl(x_{c_j+1}e(\nu)\bigr)^l=0\biggr\} .
$$
We define $$
\NH_{\mathbf{b}}^{l_1,\cdots,l_p}:=\NH_{\{1,2,\cdots,b_1\}}^{l_1}\otimes \NH_{\{c_1+1,c_1+2,\cdots,c_2\}}^{l_2}\otimes\cdots\otimes
\NH_{\{n-b_p+1,\cdots,n\}}^{l_p},
$$
where for each $1\leq j\leq p$, $\NH_{\{c_{j-1}+1,c_{j-1}+2,\cdots,c_j\}}^{l_j}$ denote the cyclotomic nilHecke algebra with standard generators $$
x_{c_{j-1}+1}, x_{c_{j-1}+2}, \cdots, x_{c_j}, \tau_{c_{j-1}+1},\tau_{c_{j-1}+2},\cdots,\tau_{c_j-1},
$$
which is isomorphic to $\NH_{b_j}^{l_j}$. Clearly, the following correspondence $$\begin{aligned}
x_{k} & \mapsto x_ke(\nu),\quad\,\forall\,1\leq k\leq n,\\
\tau_{c_{i-1}+l}  & \mapsto \tau_l e(\nu),\quad\,\forall\, 1\leq l<b_j, 1\leq i\leq p,
\end{aligned}$$
can be extended uniquely to a $K$-algebra homomorphism $\pi_{\mathbf{b},\nu}: \NH_{\mathbf{b}}^{l_1,\cdots,l_p} \rightarrow \RR^\Lam_{\alpha}$.

By construction, it is clear that $\pi_{\mathbf{b}}$ induces a homogeneous $K$-linear map: $$
\overline{\pi}_{\mathbf{b},\nu}: \Tr(\NH_{\mathbf{b}}^{l_1,\cdots,l_p}) \rightarrow \Tr(\RR^\Lam_{\alpha}).$$

As an easy application of Theorem \ref{generator}, we recover one half of \cite[Theorem 3.31(a)]{SVV} (see Corollary \ref{maincor0} for another half of \cite[Theorem 3.31(a)]{SVV}). Their original proof used the categorical representation and an action of the loop algebra, while our proof is more direct and elementary.

\begin{prop}\text{\rm (\cite[Theorem 3.31(a)]{SVV})}\label{svv1a} The cocenter $\Tr(\R[\alpha])$ of $\R[\alpha]$ is always positive graded. In other words, for any $j\in\Z$, $$
\text{$(\Tr(\R[\alpha]))_j\neq 0$\,\,\, only if \,\,\, $j\geq 0$.}
$$
Equivalently, $Z(\R[\alpha])_j\neq 0$\,\, only if \,\,$j\leq d_{\Lam,\alpha}$.
\end{prop}

\begin{proof} We consider all the $\nu\in I^\alpha$ and all the decomposition of $\nu$ as in \eqref{nu1}. Applying Theorem \ref{generator}, we can deduce that the following homogeneous $K$-linear map: $$
\pi:=\sum_{\nu,\mathbf{b}}\pi_{\mathbf{b},\nu}: \bigoplus_{\nu,\mathbf{b}}\Tr(\NH_{\mathbf{b}}^{l_1,\cdots,l_p}) \rightarrow \Tr(\RR^\Lam_{\alpha})$$
is surjective. Since $\Tr(\NH_{\mathbf{b}}^{l_1,\cdots,l_p})\cong\Tr(\NH_{b_1}^{l_1})\otimes\cdots\otimes\Tr(\NH_{b_p}^{l_p})$ is positively graded, the same must be also true for $\Tr(\R[\alpha])$. This proves the proposition.
\end{proof}

\begin{cor} Let $\alpha\in Q_n^+$. Suppose that $\R[\alpha]\neq 0$. Then $d_{\Lam,\alpha}\geq 0$.
\end{cor}

\begin{proof} By \cite[Proposition 3.10]{SVV}, there exists $0\neq h\in\R[\alpha]$ of degree $d_{\Lam,\alpha}$ such that $t_{\Lam,\alpha}(h)\neq 0$. Hence $h\notin [\R[\alpha],\R[\alpha]]$, or equivalently, $\overline{0}\neq h\in\Tr(\R[\alpha])$. Applying Proposition \ref{svv1a}, we can deduce that $d_{\Lam,\alpha}\geq 0$.
\end{proof}

Prof. Wei Hu has asked whether $d_{\Lam,\alpha}\geq 0$ is sufficient to ensure that $\R[\alpha]\neq 0$. The following example shows that this is not the case.

\begin{examp} Let $I:=\Z/3\Z$, $\fg:=\widehat{\mathfrak{sl}}_3$, $\Lam=4\Lam_0$, $\beta=\alpha_0+2\alpha_1$. Then we have $d_{\Lam,\beta}=2>0$. But $\R[\beta]=0$, because otherwise it should be a block of the cyclotomic Hecke algebra of type $G(4,1,3)$. But the latter case can not happen by \cite[Lemma 4.1]{HM} because there is no standard tableau $\t=(\t^{(1)},\t^{(2)},\t^{(3)},\t^{(4)})$ whose residues multiset is equal to $\{0,1,1\}$.
\end{examp}

In Theorem \ref{criterion} of Section 4 we shall give a necessary and sufficient condition for which $\R[\alpha]\neq 0$.

\smallskip
\subsection{Relations inside the cocenter} The purpose of this subsection is to present some relations inside the cocenter.

Let $\nu\in I^\alpha$. Let $\mathbf{b}=(b_1,\cdots,b_p)\in\mathcal{C}(\nu)$. We define $\mathbf{c}=(c_0,c_1,\cdots,c_{p})$ as in (\ref{cj1}).

\begin{dfn} Let $1\leq m,m'\leq n$. We define $A_{m}$ to be the $K$-subalgebra of $\R[\alpha]$ generated by $$
\tau_w,\, x_j,\,\, w\in \Sym_{\{1,2,\cdots,m\}}, 1\leq j\leq m .
$$
We define $B_{m'-1}$ to be the $K$-subalgebra of $\R[\alpha]$ generated by $$
\tau_w,\, x_j,\,\, w\in \Sym_{\{m',m'+1,\cdots,n\}}, m'\leq j\leq n .
$$
By convention, we set $A_0=Ke(\nu)=B_n$. We call $A_{m}$ the first $m$-th part of $\R[\alpha]$, while call $B_{m'-1}$ the last $(n-(m'-1))$-th part of $\R[\alpha]e(\nu)$.
\end{dfn}
In particular, $$\begin{aligned}
A_{m} & =\text{$K$-{\rm Span}}\Bigl\{\tau_w x_1^{k_1}x_2^{k_2}\cdots x_{m}^{k_{c_t}}e(\nu)\Bigm|w\in\Sym_{\{1.2.\cdots,m\}},\,k_1,\cdots,k_{m}\in\N\Bigr\},\\
B_{m'-1} & =\text{$K$-{\rm Span}}\Bigl\{\tau_w x_{m'}^{k_{m'}}x_{m'+1}^{k_{m'+1}}\cdots x_{n}^{k_n}e(\nu)\Bigm|\begin{matrix}\text{$w\in
\Sym_{\{m',m'+1,\cdots,n\}}$,}\\ \text{and $k_{m'},k_{m'+1},\cdots,k_n\in\N$}\end{matrix}\Bigr\}.
 \end{aligned}
$$

\begin{lem}\label{relations}
Suppose that $y=y_1 x_{c_t+1}^k\tau_{c_t+1}\tau_{c_{t}+2}\cdots \tau_{c_{t+1}-1} y_2e(\nu)$, where $k\in\N, y_1\in A_{c_t},\,y_2\in B_{c_{t+1}}, 0\leq t\leq p-1$. \begin{enumerate}
\item[(1)] If  $b_{t+1}=2$, then we have $$ \begin{aligned}
(k+1)y+&y_1\Bigl(\sum_{\substack{k_1,k_2\in\N,\\ k_1+k_2=k-1}}x_{c_t+1}^{k_1}x_{c_t+2}^{k_2}+\sum_{\substack{k_1,k'_1,k'_2\in\N,\\ 1\leq k_1\leq k-1,\\k_1'+k_2'=k-1-k_1}}x_{c_t+1}^{k_1+k_1'}x_{c_t+2}^{k_2'}\Bigr)y_2e(\nu)\in [\R[\alpha],\R[\alpha]];
\end{aligned}$$
\item[(2)] If $b_{t+1}>2$, then we have $$ \begin{aligned}
(k+1)y+&y_1\Biggl(kx_{c_t+1}^{k-1}\tau_{c_t+1}\tau_{c_{t}+2}\cdots \tau_{c_{t+1}-2}+\sum_{\substack{k_1,k'_1,k'_2\in\N,\\ 0\leq k_1\leq k-2,\\ k_1'+k_2'=k-2-k_1}}x_{c_t+1}^{k_1+k_1'}x_{c_t+2}^{k_2'}\tau_{c_{t}+2}\cdots \tau_{c_{t+1}-2} \\
&+\sum_{\substack{k_1,k'_1,k'_2\in\N,\\ 1\leq k_1\leq k-1,\\k_1'+k_2'=k-1-k_1}}x_{c_t+1}^{k_1+k_1'}x_{c_t+2}^{k_2'}\tau_{c_{t}+2}\cdots \tau_{c_{t+1}-1} \Biggr)y_2e(\nu)\in [\R[\alpha],\R[\alpha]].
\end{aligned}$$
\end{enumerate}
In particular, in both cases if $k=0$ then $y\in [\R[\alpha],\R[\alpha]]$.
\end{lem}

\begin{proof} 1) By assumption, $b_{t+1}=2$. By the decomposition of $\nu$ given in (\ref{nu1}), we can do the following calculation in $\Tr(\R[\alpha])$:
 $$\begin{aligned}
y&\equiv y_1 x_{c_{t}+1}^k\tau_{c_{t}+1}(x_{c_{t}+2}\tau_{c_{t}+1}-\tau_{c_{t}+1}x_{c_{t}+1}) y_2e(\nu)\\
&\equiv y_1 x_{c_{t}+1}^k\tau_{c_{t}+1}x_{c_{t}+2}\tau_{c_{t}+1} y_2e(\nu)\\
&\equiv y_1 \tau_{c_{t}+1}x_{c_{t}+1}^k\tau_{c_{t}+1}x_{c_{t}+2}y_2e(\nu)\qquad \text{(as $\tau_{c_t+1}=\tau_{c_{t+1}-1}$ commutes with $y_2$ and $y_1$)}\\
&\equiv -y_1 \Bigl(\sum_{\substack{k_1,k_2\in\N\\ k_1+k_2=k-1}}x_{c_{t}+1}^{k_1}x_{c_{t}+2}^{k_2}\Bigr)\tau_{c_{t}+1}x_{c_{t}+2}y_2e(\nu)\\
&\equiv -y_1 \Bigl(\sum_{\substack{k_1,k_2\in\N\\ k_1+k_2=k-1}}x_{c_{t}+1}^{k_1}x_{c_{t}+2}^{k_2}\Bigr)(1+x_{c_t+1}\tau_{c_{t}+1})y_2e(\nu)\\
&\equiv -y_1 \Bigl(\sum_{\substack{k_1,k_2\in\N\\ k_1+k_2=k-1}}x_{c_{t}+1}^{k_1}x_{c_{t}+2}^{k_2}+\sum_{\substack{k_1,k_2\in\N\\
k_1+k_2=k-1}}x_{c_{t}+1}^{k_1+1}x_{c_{t}+2}^{k_2}\tau_{c_{t}+1}\Bigr)y_2e(\nu)\\
&\equiv -y_1 \Bigl(\sum_{\substack{k_1,k_2\in\N\\ k_1+k_2=k-1}}x_{c_{t}+1}^{k_1}x_{c_{t}+2}^{k_2}+\sum_{\substack{k_1,k_2\in\N\\ k_1+k_2=k-1}}x_{c_{t}+1}^{k_1+1}\bigl(\sum_{\substack{k'_1,k'_2\in\N\\ k_1'+k_2'=k_2-1}}x_{c_{t}+1}^{k_1'}x_{c_{t}+2}^{k_2'}+
\tau_{c_{t}+1}x_{c_{t}+1}^{k_2}\bigr)\Bigr)y_2e(\nu)\\
&\equiv -y_1 \Bigl(\sum_{\substack{k_1,k_2\in\N\\ k_1+k_2=k-1}}x_{c_{t}+1}^{k_1}x_{c_{t}+2}^{k_2}+\sum_{\substack{k_1,k_2,k'_1,k'_2\in\N\\ k_1+k_2=k-1\\ k_1'+k_2'=k_2-1}}x_{c_{t}+1}^{k_1+k_1'+1}x_{c_{t}+2}^{k_2'}+
\sum_{\substack{k_1,k_2\in\N\\ k_1+k_2=k-1}}x_{c_{t}+1}^{k_1+1}\tau_{c_{t}+1}x_{c_{t}+1}^{k_2}\Bigr)y_2e(\nu)\\
&\equiv -y_1 \Bigl(\sum_{\substack{k_1,k_2\in\N\\ k_1+k_2=k-1}}x_{c_{t}+1}^{k_1}x_{c_{t}+2}^{k_2}
+\sum_{\substack{k_1,k'_1,k'_2\in\N\\ 0\leq k_1\leq k-1\\ k_1'+k_2'=k-k_1-2}}x_{c_{t}+1}^{k_1+k_1'+1}x_{c_{t}+2}^{k_2'}+
\sum_{\substack{k_1,k_2\in\N\\ k_1+k_2=k-1}}x_{c_{t}+1}^{k_2}x_{c_{t}+1}^{k_1+1}\tau_{c_{t}+1}\Bigr)y_2e(\nu)\\
&\equiv -y_1 \Bigl(\sum_{\substack{k_1,k_2\in\N\\ k_1+k_2=k-1}}x_{c_{t}+1}^{k_1}x_{c_{t}+2}^{k_2}+\sum_{\substack{k_1,k'_1,k'_2\in\N\\ 0\leq k_1\leq k-1\\ k_1'+k_2'=k-k_1-2}}x_{c_{t}+1}^{k_1+k_1'+1}x_{c_{t}+2}^{k_2'}\Bigr)y_2e(\nu)-ky\pmod{[\R[\alpha],\R[\alpha]]},
\end{aligned}
$$
where we have used the fact that $x_{c_{t}+1}^{k_2}=x_{c_{t+1}-1}^{k_2}$ commutes with $y_2$ and $y_1$ in the second last equality. This proves (1). In particular, this implies that $y_1 \tau_{c_{t}+1} y_2e(\nu)\in[\R[\alpha],\R[\alpha]]$.

2) By assumption, $b_{t+1}>2$. By the decomposition of $\nu$ given in (\ref{nu1}), we can do the following calculation in $\Tr(\R[\alpha])$:
$$\begin{aligned}
y&\equiv y_1 x_{c_{t}+1}^k\tau_{c_{t}+1}\tau_{c_{t}+2}\cdots \tau_{c_{t+1}-1}\bigl(x_{c_{t+1}}\tau_{c_{t+1}-1}-\tau_{c_{t+1}-1}x_{c_{t+1}-1}\bigr) y_2e(\nu)\\
&\equiv y_1 x_{c_{t}+1}^k\tau_{c_{t}+1}\tau_{c_{t}+2}\cdots \tau_{c_{t+1}-1}x_{c_{t+1}}\tau_{c_{t+1}-1}y_2e(\nu)\\
&\equiv y_1 \tau_{c_{t+1}-1}x_{c_{t}+1}^k\tau_{c_{t}+1}\tau_{c_{t}+2}\cdots \tau_{c_{t+1}-1}x_{c_{t+1}} y_2e(\nu)\quad \text{(as $\tau_{c_{t+1}-1}$ commutes with $y_2$ and $y_1$)}\\
&\equiv y_1x_{c_{t}+1}^k\tau_{c_{t}+1}\tau_{c_{t}+2}\cdots \tau_{c_{t+1}-3}\tau_{c_{t+1}-1}\tau_{c_{t+1}-2}\tau_{c_{t+1}-1}x_{c_{t+1}} y_2e(\nu)\\
&\equiv y_1x_{c_{t}+1}^k\tau_{c_{t}+1}\tau_{c_{t}+2}\cdots \tau_{c_{t+1}-3}\tau_{c_{t+1}-2}\tau_{c_{t+1}-1}\tau_{c_{t+1}-2}x_{c_{t+1}} y_2e(\nu)\\
&\equiv y_1\tau_{c_{t+1}-2}x_{c_{t}+1}^k\tau_{c_{t}+1}\tau_{c_{t}+2}\dots \tau_{c_{t+1}-3}\tau_{c_{t+1}-2}\tau_{c_{t+1}-1}x_{c_{t+1}} y_2e(\nu)\quad \text{(as $\tau_{c_{t+1}-2}$ commutes with $x_{c_{t+1}}, y_2, y_1$)}
\end{aligned} $$
We repeat the previous argument with $\tau_{c_{t+1}-1}$ replaced by $\tau_{c_{t+1}-2},\cdots,\tau_{c_t+3}$, and so on. Eventually, we shall get that
$$\begin{aligned}
y&\equiv y_1\tau_{c_{t}+2}x_{c_{t}+1}^k\tau_{c_{t}+1}\tau_{c_{t}+2}\cdots \tau_{c_{t+1}-2}\tau_{c_{t+1}-1}x_{c_{t+1}} y_2e(\nu)\\
&\equiv y_1x_{c_{t}+1}^k\tau_{c_{t}+2}\tau_{c_{t}+1}\tau_{c_{t}+2}\tau_{c_{t}+3}\cdots \tau_{c_{t+1}-2}\tau_{c_{t+1}-1}x_{c_{t+1}} y_2e(\nu)\\
&\equiv y_1x_{c_{t}+1}^k\tau_{c_{t}+1}\tau_{c_{t}+2}\tau_{c_{t}+1}\tau_{c_{t}+3}\cdots \tau_{c_{t+1}-2}\tau_{c_{t+1}-1}x_{c_{t+1}} y_2e(\nu)\\
&\equiv y_1x_{c_{t}+1}^k\tau_{c_{t}+1}\tau_{c_{t}+2}\tau_{c_{t}+3}\cdots \tau_{c_{t+1}-2}\tau_{c_{t+1}-1}x_{c_{t+1}} y_2e(\nu)\tau_{c_{t}+1}\\
&\equiv y_1\tau_{c_{t}+1}x_{c_{t}+1}^k\tau_{c_{t}+1}\tau_{c_{t}+2}\cdots \tau_{c_{t+1}-2}\tau_{c_{t+1}-1}x_{c_{t+1}} y_2e(\nu)\pmod{[\R[\alpha],\R[\alpha]]},
\end{aligned}
$$
Hence, if $k=0$, then we can deduce  $$y\equiv y_1\tau_{c_{t}+1}\tau_{c_{t}+1}\tau_{c_{t}+1+1}\cdots \tau_{c_{t+1}-2}\tau_{c_{t+1}-1}x_{c_{t+1}} y_2e(\nu)\equiv 0\pmod{[\R[\alpha],\R[\alpha]]}.
$$

Now suppose $k>0$. We can deduce$$\begin{aligned}
y&\equiv y_1\bigl(\tau_{c_{t}+1}x_{c_{t}+1}^k\bigr)\tau_{c_{t}+1}\tau_{c_{t}+2}\cdots \tau_{c_{t+1}-2}\tau_{c_{t+1}-1}x_{c_{t+1}} y_2e(\nu)\\
&\equiv -y_1\Bigl(\sum_{\substack{k_1,k_2\in\N\\ k_1+k_2=k-1}}x_{c_{t}+1}^{k_1}x_{c_{t}+2}^{k_2}\Bigr)\tau_{c_{t}+1}\tau_{c_{t}+2}\cdots \tau_{c_{t+1}-2}\tau_{c_{t+1}-1}x_{c_{t+1}} y_2e(\nu)\\
&\equiv -y_1\Bigl(\sum_{\substack{k_1,k_2\in\N\\ k_1+k_2=k-1}}x_{c_{t}+1}^{k_1}x_{c_{t}+2}^{k_2}\Bigr)\tau_{c_{t}+1}\tau_{c_{t}+2}\cdots \tau_{c_{t+1}-2}\bigl(x_{c_{t+1}-1}\tau_{c_{t+1}-1}+1) y_2e(\nu\bigr)\\
&\equiv -y_1\Bigl(\sum_{\substack{k_1,k_2\in\N\\ k_1+k_2=k-1}}x_{c_{t}+1}^{k_1}x_{c_{t}+2}^{k_2}\Bigr)\tau_{c_{t}+1}\tau_{c_{t}+2}\cdots \tau_{c_{t+1}-2} y_2e(\nu)\\
&\quad -y_1\Bigl(\sum_{\substack{k_1,k_2\in\N\\ k_1+k_2=k-1}}x_{c_{t}+1}^{k_1}x_{c_{t}+2}^{k_2}\Bigr)\tau_{c_{t}+1}\tau_{c_{t}+2}\cdots \tau_{c_{t+1}-3}(\tau_{c_{t+1}-2}x_{c_{t+1}-1})\tau_{c_{t+1}-1} y_2e(\nu)\\
&\equiv -y_1\Bigl(\sum_{\substack{k_1,k_2\in\N\\ k_1+k_2=k-1}}x_{c_{t}+1}^{k_1}x_{c_{t}+2}^{k_2}\Bigr)\tau_{c_{t}+1}\tau_{c_{t}+2}\cdots \tau_{c_{t+1}-2} y_2e(\nu)\\
&\quad -y_1\Bigl(\sum_{\substack{k_1,k_2\in\N\\ k_1+k_2=k-1}}x_{c_{t}+1}^{k_1}x_{c_{t}+2}^{k_2}\Bigr)\tau_{c_{t}+1}\tau_{c_{t}+2}\cdots x_{c_{t+1}-2}\tau_{c_{t+1}-2}\tau_{c_{t+1}-1} y_2e(\nu)\\
&\quad -y_1\Bigl(\sum_{\substack{k_1,k_2\in\N\\ k_1+k_2=k-1}}x_{c_{t}+1}^{k_1}x_{c_{t}+2}^{k_2}\Bigr)\tau_{c_{t}+1}\tau_{c_{t}+2}\cdots \tau_{c_{t+1}-3}\tau_{c_{t+1}-1} y_2e(\nu)\pmod{[\R[\alpha],\R[\alpha]]}.
\end{aligned}
$$
Now, for the last term above, $$
-y_1\Bigl(\sum_{\substack{k_1,k_2\in\N\\ k_1+k_2=k-1}}x_{c_{t}+1}^{k_1}x_{c_{t}+2}^{k_2}\Bigr)\tau_{c_{t}+1}\tau_{c_{t}+2}\cdots \tau_{c_{t+1}-3}e(\nu)\in A_{c_{t+1}-2},\quad y_2e(\nu)\in B_{c_{t+1}}.
$$
It follows from the $k=0$ case in the part (1) of the lemma that the last term must vanish in $\Tr(\R[\alpha])$. Therefore, $$\begin{aligned}
y&\equiv -y_1\bigl(\sum_{\substack{k_1,k_2\in\N\\ k_1+k_2=k-1}}x_{c_{t}+1}^{k_1}x_{c_{t}+2}^{k_2}\bigr)\tau_{c_{t}+1}\tau_{c_{t}+2}\cdots \tau_{c_{t+1}-2} y_2e(\nu)\\
&\quad -y_1\bigl(\sum_{\substack{k_1,k_2\in\N\\ k_1+k_2=k-1}}x_{c_{t}+1}^{k_1}x_{c_{t}+2}^{k_2}\bigr)\tau_{c_{t}+1}\tau_{c_{t}+2}\cdots \tau_{c_{t+1}-3}x_{c_{t+1}-2}\tau_{c_{t+1}-2}\tau_{c_{t+1}-1} y_2e(\nu)\\
&\equiv -y_1\bigl(\sum_{\substack{k_1,k_2\in\N\\ k_1+k_2=k-1}}x_{c_{t}+1}^{k_1}x_{c_{t}+2}^{k_2}\bigr)\tau_{c_{t}+1}\tau_{c_{t}+2}\cdots \tau_{c_{t+1}-2} y_2e(\nu)\\
&\quad-y_1\bigl(\sum_{\substack{k_1,k_2\in\N\\ k_1+k_2=k-1}}x_{c_{t}+1}^{k_1}x_{c_t+2}^{k_2}\bigr)\tau_{c_{t}+1}\tau_{c_{t}+1+1}\cdots \tau_{c_{t+1}-4}\bigl(\tau_{c_{t+1}-3}x_{c_{t+1}-2}\bigr)\tau_{c_{t+1}-2}\tau_{c_{t+1}-1} y_2e(\nu)\\
&\equiv -y_1\bigl(\sum_{\substack{k_1,k_2\in\N\\ k_1+k_2=k-1}}x_{c_{t}+1}^{k_1}x_{c_{t}+2}^{k_2}\bigr)\tau_{c_{t}+1}\tau_{c_{t}+2}\cdots \tau_{c_{t+1}-2} y_2e(\nu)\\
&\quad-y_1\bigl(\sum_{\substack{k_1,k_2\in\N\\ k_1+k_2=k-1}}x_{c_{t}+1}^{k_1}x_{c_t+2}^{k_2}\bigr)\tau_{c_{t}+1}\tau_{c_{t}+2}\cdots x_{c_{t+1}-3}\tau_{c_{t+1}-3}\tau_{c_{t+1}-2}\tau_{c_{t+1}-1} y_2e(\nu)\\
&\quad -y_1\bigl(\sum_{\substack{k_1,k_2\in\N\\ k_1+k_2=k-1}}x_{c_{t}+1}^{k_1}x_{c_{t}+2}^{k_2}\bigr)\tau_{c_{t}+1}\tau_{c_{t}+2}\cdots \tau_{c_{t+1}-4}\tau_{c_{t+1}-2}\tau_{c_{t+1}-1} y_2e(\nu)\pmod{[\R[\alpha],\R[\alpha]]}.
\end{aligned}
$$
Again, the $k=0$ case in the part (2) of the lemma implies that the last term above must vanish in $\Tr(\R[\alpha])$. We repeat the same argument above and eventually we shall get that  $$\begin{aligned}
y&\equiv -y_1\bigl(\sum_{\substack{k_1,k_2\in\N\\ k_1+k_2=k-1}}x_{c_{t}+1}^{k_1}x_{c_{t}+2}^{k_2}\bigr)\tau_{c_{t}+1}\tau_{c_{t}+2}\cdots \tau_{c_{t+1}-2} y_2e(\nu)\\
&\quad -y_1\bigl(\sum_{\substack{k_1,k_2\in\N\\ k_1+k_2=k-1}}x_{c_{t}+1}^{k_1+1}x_{c_t+2}^{k_2}\bigr)\tau_{c_{t}+1}\tau_{c_{t}+2}\cdots\tau_{c_{t+1}-1} y_2e(\nu)\pmod{[\R[\alpha],\R[\alpha]]}.
\end{aligned}
$$

Finally, we have $$\begin{aligned} &\quad\,-y_1\bigl(\sum_{k_1+k_2=k-1}x_{c_{t}+1}^{k_1}x_{c_{t}+2}^{k_2}\bigr)\tau_{c_{t}+1}\tau_{c_{t}+2}\cdots \tau_{c_{t+1}-2} y_2e(\nu)\\
&\equiv -y_1\Bigl(\sum_{\substack{k_1,k_2\in\N\\ k_1+k_2=k-1}}x_{c_{t}+1}^{k_1}\bigl(\tau_{c_{t}+1}x_{c_{t}+1}^{k_2}+\sum_{\substack{k'_1,k'_2\in\N\\ k_1'+k_2'=k_2-1}}x_{c_{t}+1}^{k_1'}x_{c_t+2}^{k_2'}\bigr)\Bigr)\tau_{c_{t}+2}
\cdots \tau_{c_{t+1}-2} y_2e(\nu)\\
&\equiv -y_1\Bigl(\sum_{\substack{k_1,k_2\in\N\\ k_1+k_2=k-1}}x_{c_{t}+1}^{k_1+k_2}\tau_{c_{t}+1}+\sum_{\substack{k_1,k_2,k'_1,k'_2\in\N\\ k_1+k_2=k-1\\k_1'+k_2'=k_2-1}}x_{c_{t}+1}^{k_1+k_1'}x_{c_{t}+2}^{k_2'}\Bigr)
\tau_{c_{t}+2}\cdots \tau_{c_{t+1}-2} y_2e(\nu)\\
&\equiv -ky_1 x_{c_{t}+1}^{k-1}\tau_{c_{t}+1}\tau_{c_{t}+2}\cdots \tau_{c_{t+1}-2} y_2e(\nu)-y_1\bigl(\sum_{\substack{k_1,k'_1,k'_2\in\N,\\ 0\leq k_1\leq k-1,\\ k_1'+k_2'=k-k_1-2}}x_{c_{t}+1}^{k_1+k_1'}x_{c_{t}+2}^{k_2'}\bigr)\tau_{c_{t}+2}\cdots \tau_{c_{t+1}-2} y_2e(\nu)\\
&\qquad\quad\pmod{[\R[\alpha],\R[\alpha]]},
\end{aligned}
$$
where we have used the fact that $x_{c_{t}+1}^{k_2}$ commutes with $\tau_{c_{t}+2}
\cdots \tau_{c_{t+1}-2}y_2e(\nu), y_1$  and moved $x_{c_{t}+1}^{k_2}$ from the right end to the left end in the second equality.

Similarly, we have  $$\begin{aligned}
&\quad\,-y_1\Bigl(\sum_{\substack{k_1,k_2\in\N\\ k_1+k_2=k-1}}x_{c_{t}+1}^{k_1+1}x_{c_{t}+2}^{k_2}\Bigr)\tau_{c_{t}+1}\tau_{c_{t}+2}\cdots\tau_{c_{t+1}-1} y_2e(\nu)\\
&\equiv -y_1\Bigl(\sum_{\substack{k_1,k_2\in\N\\ k_1+k_2=k-1}}x_{c_{t}+1}^{k_1+1}\bigl(\tau_{c_{t}+1}x_{c_{t}+1}^{k_2}+\sum_{\substack{k_1,k_2,k'_1,k'_2\in\N\\ k_1'+k_2'=k_2-1}}x_{c_{t}+1}^{k_1'}x_{c_{t}+2}^{k_2'}\bigr)\Bigr)
\tau_{c_{t}+2}\cdots \tau_{c_{t+1}-1} y_2e(\nu)\\
&\equiv -y_1\Bigl(\sum_{\substack{k_1,k_2\in\N\\ k_1+k_2=k-1}}x_{c_{t}+1}^{k_1+k_2+1}\tau_{c_{t}+1}+\sum_{\substack{k_1,k'_1,k'_2\in\N\\ 0\leq k_1\leq k-1,\\ k_1'+k_2'=k-k_1-2}}x_{c_{t}+1}^{k_1+k_1'+1}x_{c_{t}+2}^{k_2'}\Bigr)
\tau_{c_{t}+2}\cdots \tau_{c_{t+1}-1} y_2e(\nu)\\
&\equiv -ky-y_1(\sum_{\substack{k_1,k'_1,k_2\in\N\\ 0\leq k_1\leq k-1,\\ k_1'+k_2'=k-k_1-2}}x_{c_{t}+1}^{k_1+k_1'+1}x_{c_{t}+2}^{k_2'})\tau_{c_{t}+2}\cdots \tau_{c_{t+1}-1} y_2e(\nu)\pmod{[\R[\alpha],\R[\alpha]]},
\end{aligned}
$$
where we have used the fact that $x_{c_{t}+1}^{k_2}$ commutes with $\tau_{c_{t}+2}
\cdots \tau_{c_{t+1}-1}y_2e(\nu), y_1$  and moved $x_{c_{t}+1}^{k_2}$ from the right end to the left end in the second equality.
Now the Lemma follows from the last two paragraphs.
\end{proof}

\begin{cor}\label{zero element}
Suppose $\cha K=0$. Let $y=y_1 x_{c_{t}+1}^k\tau_{c_{t}+1}\tau_{c_{t}+2}\cdots \tau_{c_{t+1}-1} y_2e(\nu)$, where $y_1\in A_{c_t}, y_2\in B_{c_{t+1}}$. If $k<b_{t+1}-1$, then $y\in [\R[\alpha],\R[\alpha]]$.
\end{cor}

\begin{proof} Since $\cha K=0$, it follows that $(k+1)\cdot 1_K$ is invertible in $K$. If $k=0$, then the corollary follows from Lemma \ref{relations}.
In general, Lemma \ref{relations} gives an algorithm to rewrite the element $y$ with smaller $b_{t+1}$ and $k$. So the corollary follows from an induction on $k$.
\end{proof}

\bigskip

\section{Piecewise dominant sequence and maximal degree and minimal degree components of the cocenter}

The purpose of this section is to give the three main results of this paper. We shall introduce a new notion called ``piecewise dominant sequence'' and use it together with the main result in previous section to construct $K$-linear generators of both the maximal degree component and the minimal degree component of the cocenter $\Tr(\R[\alpha])$. In particular, we shall derive a new and simple criterion for which $\R[\alpha]\neq 0$.

Let $\nu\in I^\alpha$. There is a unique decomposition of $\nu$ as follows: \begin{equation}\label{decomp4}
\nu=(\nu_1,\cdots,\nu_n)=(\underbrace{\nu^{1},\nu^{1},\cdots,\nu^{1}}_{b_{1}\,copies},\cdots,\underbrace{\nu^{p},\nu^{p},\cdots,\nu^{p}}_{b_{p}\,copies}),
\end{equation}
which satisfies that \begin{equation}\label{newrequire}
\nu^{j}\neq\nu^{j+1},\quad \forall\, 1\leq j<p,
\end{equation}
where $p,b_1,\cdots,b_p\in\Z^{\geq 1}$ with $\sum_{i=1}^{p}b_i=n$.
%For each $1\leq j\leq n$, we define \begin{equation}\label{Elli30}
%L_j(\nu):=\<h_{\nu_j}, \Lam-\sum_{t=1}^{j-1}\alpha_{\nu_t}\>.
%\end{equation}
For each $1\leq i\leq p$, we define \begin{equation}\label{elli30}
\ell_i(\nu):=\<h_{\nu^i}, \Lam-\sum_{j=1}^{c_{i-1}}\alpha_{\nu_j}\>,
\end{equation}
where $\{c_j|0\leq j\leq p\}$ is as defined in (\ref{cj1}). %Throughout this section, we shall fix the above decomposition of $\nu$.
When $\nu$ is clear from the context, we shall write $\ell_i$ instead of $\ell_i(\nu)$ for simplicity.

\subsection{Piecewise dominant sequence}

%The following definition is important for this controls the dimension of the $\overline{\RR^\Lam_\alpha}^{d_{\Lam,\alpha}}$.

\begin{dfn}\label{pddfn} Let $\Lam\in P^+$ and $\alpha\in Q_n^+$. We call $\nu=(\nu_1,\cdots,\nu_n)\in I^\alpha$ a {\bf piecewise dominant sequence} with respect to $\Lam$, if for the unique decomposition (\ref{decomp4}) of $\nu$ and any $1\leq i\leq p$, \begin{equation}\label{elli}
\ell_i=\ell_i(\nu)\geq b_i.
\end{equation}
\end{dfn}

%for any $0\leq i\leq p-1$, and any $c_i+1\leq k\leq c_{i+1}$, $$\<h_{\nu_k}, \Lam-\sum_{j=1}^{k-1}\alpha_{\nu_j}\>\geq c_{i+1}-k.
%$$
\begin{examp}\label{examp1} Consider the cyclotomic nilHecke algebra $\NH_n^\ell$ of type $A$. Let $\nu=(\underbrace{0,0,\cdots,0}_{n\,copies})$. Then an easy computation shows that $\nu$ is piecewise dominant if and only if $\ell\geq n$, i.e., $\NH_n^\ell\neq 0$.
\end{examp}

\begin{lem}\label{princlple criterion} Let $\nu=(\nu_1,\cdots,\nu_n)\in I^\alpha$ and fix the unique decomposition (\ref{decomp4}) of $\nu$. Then $\nu$ is a piecewise dominant sequence with respect to $\Lam$ if and only if for each $1\leq i\leq p$, there is an integer $c_{i-1}+1\leq k'_i\leq c_{i}$ such that
\begin{equation}\label{pdCondition}
\<h_{\nu_{k'_i}}, \Lam-\sum_{j=1}^{k'_i-1}\alpha_{\nu_j}\>\geq c_i-k'_i+1.\end{equation}
In this case, we denote the maximal value of each $k'_i$ by $k_i$, which can be taken as: \begin{equation}\label{maxki}
k_i=\begin{cases} c_i, &\text{if $\ell_i-2b_i\geq 0$;}\\
\ell_i+2c_{i-1}-c_{i}+1, &\text{if $\ell_i-2b_i\leq -1$.}
\end{cases}
\end{equation}
\end{lem}

\begin{proof} Suppose that $\nu$ is a piecewise dominant sequence with respect to $\Lam$ with a unique decomposition as in (\ref{decomp4}). Then for each $1\leq i\leq p$, we can simply take $k'_i:=c_{i-1}+1$. This proves one direction.

Conversely, suppose that for the unique decomposition of $\nu$ as in (\ref{decomp4}) and for each $1\leq i\leq p$, there is an integer $c_{i-1}+1\leq k'_i\leq c_{i}$ such that (\ref{pdCondition}) holds. Then we take $k_i$ to be the maximal value of $k'_i$ such that (\ref{maxki}) holds. Let $1\leq i\leq p$. Recall that $$
\ell_i=\ell_i(\nu)=\<h_{\nu^i}, \Lam-\sum_{j=1}^{i-1}b_j\alpha_{\nu^j}\> .
$$
If $\ell_i-2b_i \geq 0$, then we can take $k_i=c_{i}$ such that (\ref{pdCondition}) holds. If $\ell_i-2b_i\leq -1$, then it is easy to see that
the following inequalities :\[
\begin{cases}
\ell_i-2(k_i-c_{i-1})\geq c_{i}-1-k_i\\
\ell_i-2(k_i+1-c_{i-1})<c_{i}-1-(k_i+1) \\
c_{i-1}+1\leq k_i\leq c_{i}.\\
\end{cases}\] has a unique solution $k_i:=\ell_i+2c_{i-1}-c_{i}+1$. Combining this with the third inequality we can deduce that $\ell_i\geq b_i$, we prove that $\nu$ is a piecewise dominant sequence with respect to $\Lam$.
\end{proof}

%Suppose $\nu$ is a piecewise dominant sequence with respect to $\Lam$. Then for each $1\leq i\leq p$, we can choose a maximal integer $c_{i-1}+1\leq k_i\leq c_{i}$ such that (\ref{pdCondition}) holds. In the previous example (\ref{examp1}) for $\NH_n^\ell$. If $\ell\geq 2n$, then $k_1=n$; while if $n\leq \ell\leq 2n-1$, then $k_1=\ell-n+1$.

\begin{dfn}\label{znudfn} Let $\nu\in I^\alpha$ be a piecewise dominant sequence with the unique decomposition as in (\ref{decomp4}). We define $$
Z(\nu)=Z(\nu)_1Z(\nu)_2\cdots Z(\nu)_p, $$
where for each $1\leq i\leq p$, $$
Z(\nu)_i:=\begin{cases}
x_{c_{i-1}+1}^{\ell_i-1}
x_{c_{i-1}+2}^{\ell_i-3}\cdots  x_{c_i}^{\ell_i-2b_i+1}e(\nu), &\text{if $\ell_i\geq 2b_i$;}\\
x_{c_{i-1}+1}^{\ell_i-1}
x_{c_{i-1}+2}^{\ell_i-3}\cdots  x_{\ell_i+2c_{i-1}-c_i}^{2b_i-\ell_i+1}e(\nu), &\text{if $b_i<\ell_i\leq 2b_i-1$;}\\
e(\nu), &\text{if $\ell_i=b_i$.}
\end{cases}
$$
\end{dfn}

%\begin{thm}
%The non-zero generators in the cocenter $\RR^\Lam(\alpha)/[\RR^\Lam(\alpha),\RR^\Lam(\alpha)]$ of maximal degree are in the set $$\{Z(\nu)|\text{$\nu$ is piecewise dominant}\}.
%$$
%\end{thm}
%
%
%\begin{proof}
%Recall that char $\textbf{K}$=0. From Theorem \ref{generator}, and Corollary \ref{zero element}(1), we can deduce the following fact:
%\end{proof}

\begin{lem}\label{degree1} Suppose $\nu\in I^\alpha$ is a piecewise dominant sequence with respect to $\Lam$, then $\deg(Z(\nu))=d_{\Lam,\alpha}$.
\end{lem}

\begin{proof} Let $\nu$ be a piecewise dominant sequence with respect to $\Lam$ with a decomposition as in (\ref{decomp4}) which satisfies (\ref{newrequire}).
There are two cases:

\smallskip
{\it Case 1.} $\ell_i-2b_i\geq 0$. In this case, by definition, $$
Z(\nu)_i=x_{c_{i-1}+1}^{\ell_i-1}x_{c_{i-1}+2}^{\ell_i-3}\cdots x_{c_{i}}^{\ell_i-2b_i+1}e(\nu).
$$
 A direct computation shows $\deg(Z(\nu)_i)=(\alpha_{\nu^i},\alpha_{\nu^i})(\ell_i-b_i)b_i$.

\smallskip
{\it Case 2.} $\ell_i=b_i$. In this case, $\deg(Z(\nu)_i)=\deg(e(\nu))=0$.

\smallskip
{\it Case 3.} $b_i<\ell_i-2b_i\leq -1$. In this case, by definition, $$\begin{aligned}
\deg(Z(\nu)_i)&=(\alpha_{\nu^i},\alpha_{\nu^i})\bigl(\ell_i-1+\ell_i-3+\cdots+2b_i-\ell_i+1\bigr)\\
&=(\alpha_{\nu^i},\alpha_{\nu^i})(\ell_i-b_i)b_i.
\end{aligned}
$$
In both cases we have $$\begin{aligned}
d_{\Lam,\alpha}&=2(\alpha,\Lam)-(\alpha,\alpha)\\
&=\sum_{i=1}^p(\alpha_{\nu^i},\alpha_{\nu^i})b_i\<h_{\nu^i},\Lam\>-
2\sum_{i=1}^p\sum_{1\leq j<i}b_ib_j(\alpha_{\nu^i},\alpha_{\nu^j})-\sum_{i=1}^pb_i^2 (\alpha_{\nu^i},\alpha_{\nu^i})\\
&=\sum_{i=1}^{p}(\alpha_{\nu^i},\alpha_{\nu^i})(\ell_i-b_i)b_i=\sum_{i=1}^{p}\deg(Z(\nu)_i)=\deg(Z(\nu)).
\end{aligned}
$$
This proves the lemma.
\end{proof}

\begin{lem}\label{refineresidue}
Suppose $\cha K=0$. Let $\nu\in I^\alpha$ and $z\in \RR^\Lam_{\nu,1}$. If $\nu$ is not piecewise dominant with respect to $\Lam$, then $z\in[\R[\alpha],\R[\alpha]]$.
\end{lem}

\begin{proof}
Applying Theorem \ref{generator}, we see that $\overline{z}$ is a linear combination of the image of some elements of the form $$
\bigl(x_1^{l_1}\tau_1\tau_2\cdots\tau_{c_{1}-1}\bigr)\bigl(x_{c_{1}+1}^{l_{c_{1}+1}}\tau_{c_{1}+1}\tau_{c_{1}+2}\cdots\tau_{c_{2}-1}\bigr)\cdots
 \bigl(x_{c_{p-1}+1}^{l_{c_{p-1}+1}}\tau_{c_{p-1}+1}\cdots\tau_{n-1}\bigr)e(\nu),
$$
where $\mathbf{b}=(b_1,\cdots,b_p)\in\mathcal{C}^\Lam(\nu)$ and $\{c_j|1\leq j\leq p\}$ is as defined in (\ref{cj1}). We decompose $\nu\in I^\alpha$ as in (\ref{nu1}). Then $$
0\leq l_j\leq\<h_{\nu_j},\Lam-\sum_{k=1}^{j-1}\alpha_{\nu_k}\>-1,\quad\forall\,j\in\{1,c_1+1,\cdots,c_{p-1}+1\}.
$$ Now Corollary \ref{zero element} tells us the above element is non-zero in the cocenter only if the following holds: $$
l_1\geq c_1-1,\,\,l_{c_{1}+1}\geq c_2-c_1-1,\,\,\cdots,\,\,l_{c_{p-1}+1}\geq n-1-c_{p-1}.
$$ Applying Lemma \ref{princlple criterion},  we see that if $\nu$ is not piecewise dominant, then $\overline{z}=0$ in $\Tr(\RR^\Lam_\alpha)$.
\end{proof}

The following theorem is the first main result in this paper,

\begin{thm}\label{principle generator} Suppose $\cha K=0$. Then we have

1) $
\bigl(\Tr(\RR^\Lam_\alpha)\bigr)_{d_{\Lam,\alpha}}=\text{$K$-{\rm Span}}\Bigl\{Z(\nu)+[\R[\alpha],\R[\alpha]] \Bigm|\begin{matrix}\text{$\nu$ is piecewise dominant}\\ \text{with respect to $\Lam$}\end{matrix}\Bigr\}$;

2) $\bigl(\Tr(\RR^\Lam_\alpha)\bigr)_{0}=\text{$K$-{\rm Span}}\Bigl\{e(\nu)+[\R[\alpha],\R[\alpha]]\Bigm|\begin{matrix}\text{$\nu$ is piecewise dominant}\\
\text{with respect to $\Lam$} \end{matrix}\bigr\}$;

3) $$
\Tr(\R[\alpha])=\text{$K$-{\rm Span}}\Biggl\{x_{1}^{t_1}x_{2}^{t_2}\cdots x_{n}^{t_n}e(\nu)+[\R[\alpha],\R[\alpha]]\Biggm|\begin{matrix}\text{$t_1,t_2,\cdots,t_n\in\N,\nu\in I^\alpha$ is}\\
\text{piecewise dominant with respect  to $\Lam$,}\\
\text{ $0\leq t_j\leq\ell_i(\nu)-1, \forall\,c_{i-1}+1\leq j\leq c_i$;}\\
\text{$\forall\,1\leq i\leq p$.}  \end{matrix}\Biggr\}.
$$
\end{thm}

\begin{proof} By Theorem \ref{generator} and Lemma \ref{refineresidue}, each nonzero generator  is an image of some element of the form $$\begin{aligned}
& \bigl(x_1^{l_1}\tau_1\tau_2\cdots\tau_{c'_{1}-1}\bigr)\bigl(x_{c'_{1}+1}^{l_{c'_{1}+1}}\tau_{c'_{1}+1}\tau_{c'_{1}+2}\cdots\tau_{c'_{2}-1}\bigr)\cdots\\
&\qquad \bigl(x_{c'_{p'-1}+1}^{l_{c'_{p'-1}+1}}\tau_{c'_{p'-1}+1}\tau_{c'_{p'-1}+2}\cdots\tau_{n-1}\bigr)e(\nu),
\end{aligned}
$$ where $\nu$ is piecewise dominant. Furthermore, by Lemma \ref{princlple criterion}, for a piecewise dominant sequence $\nu$, the above element achieves the maximal degree if and only if it is of the form $$Z'(\nu)=Z'(\nu)_1Z'(\nu)_2\cdots Z'(\nu)_p, $$
where for each $1\leq i\leq p$, $$
Z'(\nu)_i:=\begin{cases}
x_{c_{i-1}+1}^{\ell_i-1}
x_{c_{i-1}+2}^{\ell_i-3}\cdots  x_{c_i}^{\ell_i-2b_i+1}e(\nu), &\text{if $\ell_i\geq 2b_i$;}\\
x_{c_{i-1}+1}^{\ell_i-1}
x_{c_{i-1}+2}^{\ell_i-3}\cdots  x_{k_i}^{\ell_i-2(k_i-c_{i-1})+1}\tau_{k_i}\cdots\tau_{c_i-2}\tau_{c_i-1}e(\nu), &\text{if $\ell_i\leq 2b_i-1$,}
\end{cases}
$$
where $k_i$ is as defined in (\ref{maxki}). Now Lemma \ref{relations} implies that $Z'(\nu)$ is some multiple of $Z(\nu)$ which exactly reach the maximal degree in cocenter by Lemma \ref{degree1}, while
the above element achieve the minimal degree if and only if it is of the form $e(\nu)$. Combining this with Lemma \ref{degree1} we complete the proof of Part 1) and Part 2) of the theorem. Part 3) follows from Theorem \ref{generator}, Lemma \ref{relations}, Corollary \ref{zero element} and Lemma \ref{refineresidue}.
\end{proof}

\begin{rem} We remark that if $\cha K>0$ then Part 3) of the above theorem may not hold. For example, by Lemma \ref{relations}, inside $\NH_2^3$ we have $$\begin{aligned}
x_1\tau_1&\equiv x_1\tau_1(x_2\tau_1-\tau_1x_1)\equiv x_1\tau_1x_2\tau_1\equiv (\tau_1x_1\tau_1)x_2\\
&\equiv -\tau_1x_2\equiv -x_1\tau_1-1\pmod{[\NH_2^3,\NH_2^3]},
\end{aligned}
$$
which implies that $2x_1\tau_1\equiv 1\pmod{[\NH_2^3,\NH_2^3]}$. Using \cite[Corollary 5.10]{HuL}, we know that $$
t_{3\Lam_0,2\alpha_0}(\tau_1x_1(x_1x_2))=1,
$$
which implies that $x_1\tau_1\notin[\NH_2^3,\NH_2^3]$, hence the degree $0$ component of the cocenter of $\NH_2^3$ is spanned by $x_1\tau_1+[\NH_2^3,\NH_2^3]$.
However, if $\cha K=2$ then $1\in [\NH_2^3,\NH_2^3]$. Hence Part 3) of the above theorem does not hold for $\NH_2^3$ in this case.
\end{rem}

Another direct application of Theorem \ref{generator} and Lemma \ref{refineresidue} is the following corollary, which recovers \cite[Theorem 3.31(a)]{SVV} in an elementary way.

\begin{cor}[\text{\rm \cite[Theorem 3.31(a)]{SVV}}]\label{maincor0} Suppose $\cha K=0$. Then we have $\bigl(\Tr(\R[\alpha])\bigr)_j\neq 0$ only if $j\in [0,d_{\Lam,\alpha}]$.
\end{cor}

\begin{proof} Suppose $\bigl(\Tr(\R[\alpha])\bigr)_j\neq 0$. Then $j\geq 0$ by Proposition \ref{svv1a}. Now $j\leq d_{\Lam,\alpha}$ follows from the proof of
Theorem \ref{principle generator}.
\end{proof}

\begin{rem} It is tempting to speculate that $\{e(\nu)|\text{$\nu$ is piecewise dominant}\}$ is a basis of $\bigl(\Tr(\R[\alpha])\bigr)_0$. Unfortunately, this is not true. For example, in the type $A_2$ case we choose $\Lam=\Lam_1+\Lam_2$. We can write down all the piecewise dominant sequences with respect to $\Lam$ as follows (where $\alpha$s are given below):
$$\begin{aligned}
&\alpha=0:\emptyset;\quad  \alpha_1:(1);\quad  \alpha_2:(2);\quad \alpha_1+\alpha_2: (1,2),(2,1); \quad \alpha_1+2\alpha_2:(1,2,2); \\
&2\alpha_1+\alpha_2:(2,1,1);\quad 2\alpha_1+2\alpha_2:(2,1,1,2), (1,2,2,1).
\end{aligned}
$$ However, by the end of proof in \cite[Theorem 3.31(c)]{SVV}, we have $$
\dim\,\bigl(\Tr(\R[\alpha])\bigr)_0=\dim\,V(\Lam)=8<9. $$
\end{rem}

\subsection{A criterion for $\RR^\Lam_\alpha\neq 0$}
In this subsection, we will give a criterion for which $\R[\alpha]\neq 0$ via the existence of piecewise dominant sequences. Throughout this subsection, unless otherwise stated, $K$ is a field of arbitrary characteristic.

\begin{dfn} Let $\mathbf{b}:=(b_1,\cdots,b_p)$ be a composition of $n$.  We define $$ \Sym_{\mathbf{b}}=\Sym_{\{1,2,\cdots,b_1\}}\times\Sym_{\{b_1+1,\cdots,b_1+b_2\}}\times\cdots\times\Sym_{\{n-b_p+1,\cdots,n\}}.
$$
which is the standard Young subgroup of $\Sym_n$ corresponding to $\mathbf{b}:=(b_1,\cdots,b_p)$.
\end{dfn}

For each composition $\mathbf{b}=(b_1,\cdots,b_p)$ of $n$, we denote by $w_{\mathbf{b},0}$ the unique longest element in $\Sym_{\mathbf{b}}$. In other words, $$
w_{\mathbf{b},0}=w_{b_1,0}^{(1)}\times w_{b_2,0}^{(2)}\times\cdots\times w_{b_p,0}^{(p)},$$
where for each $1\leq i\leq p$, $w_{b_i,0}^{(0)}$ is the unique longest element in the Young subgroup $$
\Sym_{\{b_1+\cdots+b_{i-1}+1,b_1+\cdots+b_{i-1}+2,\cdots,b_1+\cdots+b_i\}}. $$
For each $1\leq i\leq p$, we fix a reduced expression of $w_{b_i,0}^{(0)}$ and use it to define $\tau_{\mathbf{b},i}:=\tau_{w_{b_i,0}^{(0)}}$.

\begin{dfn} Let $\nu$ be a piecewise dominant sequence with respect to $\Lam$, with a decomposition (\ref{decomp4}) satisfying (\ref{newrequire}). Let
$\{c_i|0\leq i\leq p\}$ and $\{\ell_i|1\leq i\leq p\}$ be defined as in (\ref{cj1}) and (\ref{elli}) respectively. We define $$
S(\nu):=\overrightarrow{\prod_{1\leq i\leq p}}\Bigl(\tau_{\mathbf{b},i}x_{c_{i-1}+1}^{\ell_i-1}
x_{c_{i-1}+2}^{\ell_i-2}\cdots x_{c_i}^{\ell_i-b_i}\Bigr)e(\nu)\in\R[\alpha] .$$
\end{dfn}

By Definition \ref{pddfn} of piecewise dominant sequence, each power index of $x_{c_{i-1}+j}$ in the product of the above big bracket is non-negative, hence the element $S(\nu)$ is well-defined.

Recall the map $\hat{\eps}_{k,\nu_k}$ introduced in Section 2 after Lemma \ref{KK1}.

\begin{lem}\label{induction}
Let $\nu=(\nu_1,\cdots,\nu_n)$ be a piecewise dominant sequence with respect to $\Lam$, with a decomposition (\ref{decomp4}) satisfying (\ref{newrequire}). Write $\nu'=(\nu_1,\cdots,\nu_{n-1})$. Then $\nu'$ is also piecewise dominant with respect to $\Lam$ and $$\hat{\eps}_{n,\nu^p}(S(\nu))=S(\nu').
$$
\begin{proof}
The first statement follows from the definition of piecewise dominant sequence. It remains to prove the second statement. Assume $b_p=1$. Then by Definition \ref{pddfn} we have $\ell_i\geq 1$ and Lemma \ref{KK1}, $$S(\nu)=s(\nu')x_n^{\ell_p-1}=p_{\ell_p-1}(S(\nu))x_n^{\ell_p-1}.
$$ Hence, $\hat{\eps}_{n,\nu^p}(S(\nu))=s(\nu')$.

Now assume $b_p>1$. We set $\mathbf{b}':=(b_1,\cdots,b_p-1)$.  Since $\hat{\eps}_{n,\nu^p}$ is $\R[n-1]$-linear, we have $$\begin{aligned}
\hat{\eps}_{n,\nu^p}(S(\nu))&=\overrightarrow{\prod_{1\leq i\leq p-1}}\Bigl(\tau_{\mathbf{b},i}x_{c_{i-1}+1}^{\ell_i-1}
x_{c_{i-1}+2}^{\ell_i-2}\cdots x_{c_i}^{\ell_i-b_i}\Bigr)e(\nu')\\
&\qquad\qquad \times \hat{\eps}_{n,\nu^p}\Bigl(\tau_{\mathbf{b},p}x_{c_{p-1}+1}^{\ell_p-1}
x_{c_{p-1}+2}^{\ell_p-2}\cdots x_n^{\ell_p-b_p}e(\nu)\Bigr).
\end{aligned}$$
It remains to show that $$\hat{\eps}_{n,\nu^p}\Bigl(\tau_{\mathbf{b},p}x_{c_{p-1}+1}^{\ell_p-1}
x_{c_{p-1}+2}^{\ell_p-2}\cdots x_n^{\ell_p-b_p}e(\nu)\Bigr)=\tau_{\mathbf{b}', p}x_{c_{p-1}+1}^{\ell_p-1}x_{c_{p-1}+2}^{\ell_p-2}\cdots x_{n-1}^{\ell_p-b_p+1}e(\nu').
$$

By a similar calculation as in the first paragraph of the proof of \cite[Lemma 5.6]{HuL}, we obtain $$\begin{aligned}
&\quad\,\tau_{\mathbf{b},p}x_{c_{p-1}+1}^{\ell_p-1}x_{c_{p-1}+2}^{\ell_p-2}\cdots x_{n}^{\ell_p-b_p}e(\nu)\\
&=\mu_{\tau_{n-1}}\Bigl(\tau_{c_{p-1}+1}\cdots \tau_{n-2}x_{n-1}^{\ell_p-b_p}e(\nu')\otimes \tau_{\mathbf{b}', p}x_{c_{p-1}+1}^{\ell_p-1}x_{c_{p-1}+2}^{\ell_p-2}\cdots x_{n-1}^{\ell_p-b_p+1}e(\nu')\Bigr)+\\
&\sum_{\substack{a_1+a_2=\ell_p-b_p-1\\a_1,a_2\geq 0}} \tau_{\mathbf{b}', p}x_{n-1}^{a_1}\tau_{n-2}\cdots \tau_{c_{p-1}+2}\tau_{c_{p-1}+1}x_{c_{p-1}+1}^{\ell_p-1}x_{c_{p-1}+2}^{\ell_p-2}\cdots x_{n-1}^{\ell_p-b_p+1}x_{n}^{a_2}e(\nu)\\
&=\mu_{\tau_{n-1}}\Bigl(\tau_{c_{p-1}+1}\cdots \tau_{n-2}x_{n-1}^{\ell_p-b_p}e(\nu')\otimes \tau_{\mathbf{b}', p}x_{c_{p-1}+1}^{\ell_p-1}x_{c_{p-1}+2}^{\ell_p-2}\cdots x_{n-1}^{\ell_p-b_p+1}e(\nu')\Bigr)+\\
&\sum_{\substack{a_1+a_2=\ell_p-b_p-1\\a_1\geq b_p-2,a_2\geq 0}} \tau_{\mathbf{b}', p}x_{n-1}^{a_1}\tau_{n-2}\cdots \tau_{c_{p-1}+2}\tau_{c_{p-1}+1}x_{c_{p-1}+1}^{\ell_p-1}x_{c_{p-1}+2}^{\ell_p-2}\cdots x_{n-1}^{\ell_p-b_p+1}x_{n}^{a_2}e(\nu),
\end{aligned}$$
where in the second equality we used the fact that  $\tau_{\mathbf{b}', p}\tau_{j}=0$ for any $c_{p-1}+1\leq j\leq n-2$, and $x_{n-1}^{a_1}\tau_{n-2}\cdots \tau_{c_{p-1}+2}\tau_{c_{p-1}+1}\in\sum_{j=c_{p-1}+1}^{n-2}\tau_j\R[\alpha]$ whenever $a_1<b_p-2$.
Now there are two possibilities:

\smallskip
{\it Case 1.} $\ell_p-2b_p > -2$, which corresponds to the case $\lam_{n-1,\nu_p}>0$ in the notation of Lemma \ref{KK1}. In that case, the map $\hat{\eps}_{n,\nu^p}$ picks out the coefficient of $x_n^{\ell_p-2b_p+1}$ (i.e., set $a_2=\ell_p-2b_p+1$). We get that $$\begin{aligned}
&\quad\,\hat{\eps}_{n,\nu^p}\Bigl(\tau_{\mathbf{b},p}x_{c_{p-1}+1}^{\ell_p-1}
x_{c_{p-1}+2}^{\ell_p-2}\cdots x_n^{\ell_p-b_p}e(\nu)\Bigr)\\
&=\tau_{\mathbf{b}', p}x_{n-1}^{b_p-2}\tau_{n-2}\cdots \tau_{c_{p-1}+2}\tau_{c_{p-1}+1}x_{c_{p-1}+1}^{\ell_p-1}x_{c_{p-1}+2}^{\ell_p-2}\cdots x_{n-1}^{\ell_p-b_p+1}e(\nu')\\&=\tau_{\mathbf{b}', p}x_{c_{p-1}+1}^{\ell_p-1}x_{c_{p-1}+2}^{\ell_p-2}\cdots x_{n-1}^{\ell_p-b_p+1}e(\nu'),
\end{aligned}$$
where the second equality follows again from the same argument used in the last sentence of the previous paragraph.

\smallskip
{\it Case 2.} $\ell_p-2b_p\leq -2$, which corresponds to the case $\lam_{n-1,\nu_p}\leq 0$ in the notation of Lemma \ref{KK1}.  Note that $$\ell_p-b_p-1-(b_p-2)=\ell_p-2b_p+1< 0.
$$
By the same argument used in the last sentence of the paragraph above Case 1, we can deduce that $$
\sum_{\substack{a_1+a_2=\ell_p-b_p-1\\a_1\geq b_p-2,a_2\geq 0}} \tau_{\mathbf{b}', p}x_{n-1}^{a_1}\tau_{n-2}\cdots \tau_{c_{p-1}+1}x_{c_{p-1}+1}^{\ell_p-1}x_{c_{p-1}+2}^{\ell_p-2}\cdots x_{n-1}^{\ell_p-b_p+1}x_{n}^{a_2}e(\nu)=0.
$$

For any $0\leq k\leq -\ell_p+2b_p-3=-(\ell_p-2b_p+2)-1$, it follows from the same argument as in the last paragraph of the proof of \cite[Lemma 5.7]{HuL},$$\begin{aligned}
&\quad\,\mu_{x_{n-1}^k} \Bigl(\tau_{c_{p-1}+1}\cdots \tau_{n-2}x_{n-1}^{\ell_p-b_p}e(\nu')\otimes \tau_{\mathbf{b}', p}x_{c_{p-1}+1}^{\ell_p-1}x_{c_{p-1}+2}^{\ell_p-2}\cdots x_{n-1}^{\ell_p-b_p+1}e(\nu')\Bigr)\\
&=\tau_{\mathbf{b}', p}x_{n-1}^{\ell_p-b_p+k}\tau_{n-2}\cdots \tau_{c_{p-1}+2}\tau_{c_{p-1}+1}x_{c_{p-1}+1}^{\ell_p-1}x_{c_{p-1}+2}^{\ell_p-2}\cdots x_{n-1}^{\ell_p-b_p+1}e(\nu')\\
&=0.
\end{aligned}$$
Hence by Lemma \ref{KK1}, $\widetilde{z}=\tau_{c_{p-1}+1}\cdots \tau_{n-2}x_{n-1}^{\ell_p-b_p}e(\nu')\otimes \tau_{\mathbf{b}', p}x_{c_{p-1}+1}^{\ell_p-1}x_{c_{p-1}+2}^{\ell_p-2}\cdots x_{n-1}^{\ell_p-b_p+1}e(\nu')$, where $$
z:=\mu_{\tau_{n-1}} \Bigl(\tau_{c_{p-1}+1}\cdots \tau_{n-2}x_{n-1}^{\ell_p-b_p}e(\nu')\otimes \tau_{\mathbf{b}', p}x_{c_{p-1}+1}^{\ell_p-1}x_{c_{p-1}+2}^{\ell_p-2}\cdots x_{n-1}^{\ell_p-b_p+1}e(\nu')\Bigr).
$$
By the definition of $\hat{\eps}_{n,\nu^p}$, we have that $$\begin{aligned}
&\quad\,\hat{\eps}_{n,\nu^p}(\tau_{\mathbf{b},p}x_{c_{p-1}+1}^{\ell_p-1}x_{c_{p-1}+2}^{\ell_p-2}\cdots x_{c_p}^{\ell_p-b_p}e(\nu))\\
&=\mu_{x_{n-1}^{-(\ell_p-2b_p+2)}}\Bigl(\tau_{c_{p-1}+1}\cdots \tau_{n-2}x_{n-1}^{\ell_p-b_p}e(\nu')\otimes \tau_{\mathbf{b}', p}x_{c_{p-1}+1}^{\ell_p-1}x_{c_{p-1}+2}^{\ell_p-2}\cdots x_{n-1}^{\ell_p-b_p+1}e(\nu')\Bigr)\\
&=\tau_{\mathbf{b}', p}x_{n-1}^{b_p-2}\tau_{n-2}\cdots \tau_{c_{p-1}+2}\tau_{c_{p-1}+1}x_{c_{p-1}+1}^{\ell_p-1}x_{c_{p-1}+2}^{\ell_p-2}\cdots x_{n-1}^{\ell_p-b_p+1}e(\nu')\\
&=\tau_{\mathbf{b}', p}x_{c_{p-1}+1}^{\ell_p-1}x_{c_{p-1}+2}^{\ell_p-2}\cdots x_{n-1}^{\ell_p-b_p+1}e(\nu').
\end{aligned}$$
This completes the proof of the lemma.
\end{proof}
\end{lem}

Recall the  symmetrizing form $t_{\Lam,\alpha}: \R[\alpha]\rightarrow K$  introduced in Lemma \ref{tLam}.

\begin{cor}\label{Pd trace} Let $\nu\in I^\alpha$ be a piecewise dominant sequence with respect to $\Lam$. Then $$
t_{\Lam,\alpha}(S(\nu))\in K^\times. $$
In particular, $0\neq S(\nu)\notin [\R[\alpha],\R[\alpha]]$ and $0\neq e(\nu)\in\R[\alpha]$.
\end{cor}

\begin{proof}
This follows from the definition of $t_{\Lam,\alpha}$ and Lemma \ref{induction}.
\end{proof}

\begin{rem} If $\nu\in I^\alpha$ satisfies the stronger assumption that $\nu^i\neq \nu^j$ for any $1\leq i\neq j\leq p$, then $\nu$ coincides with $\widetilde{\nu}$ in the notation of \cite[(5.1)]{HS}. In this special case, the number $\ell_i$ is the same as $N_i^\Lam(\widetilde{\nu})$ in the notation of
\cite[Definition 5.2]{HS}, and the second part of \cite[Theorem 5.4]{HS} can be reformulated as: $$
\text{$e(\nu)\neq 0$ in $\R[\alpha]$\,\,\, if and only if\,\,\, $\nu$ is piecewise dominant with respect to $\Lam$.} $$

Our Corollary \ref{Pd trace} says that for those $\nu$ not satisfying the above stronger assumption, the ``if part'' of the above statement still holds. But the ``only if part'' of the above statement may be false. For example, let's consider the type $A_2$ case again and choose $\Lam=\Lam_1+\Lam_2$, then \cite[Theorem 5.34]{HS} implies $e(2,1,2)\neq 0$. But it's easy to check by definition that $(2,1,2)$ is not piecewise dominant with respect to $\Lam$.
\end{rem}

The following two theorems are the second main results of this paper.

\begin{thm}\label{criterion} Let $K$ be a field of arbitrary characteristic, $\Lam\in P^+$ and $\alpha\in Q_n^+$. The following statements are equivalent:\begin{enumerate}
\item[1)] $\R[\alpha](K)\neq 0$;
\item[2)] There is a piecewise dominant sequence $\nu\in I^\alpha$ with respect to $\Lam$.
\end{enumerate}
\end{thm}

\begin{proof} Let $\R[\alpha](\Q)$ be the cyclotomic quiver Hecke algebra defined by the same Cartan datum as $\R[\alpha](K)$ but using certain polynomials $\{Q'_{ij}(u,v)|i,j\in I\}$ defined over $\Q$. By Lemma \ref{HS1}, we see that $\R[\alpha](K)\neq 0$ if and only if $\R[\alpha](\Q)\neq 0$.

Suppose now $\R[\alpha](K)\neq 0$. Then $\R[\alpha](\Q)\neq 0$. In particular, $$
\dim\Tr(\R[\alpha](\Q))_{d_{\Lam, \alpha}}=\dim Z(\R[\alpha](\Q))_0\neq 0. $$
Applying  Theorem \ref{principle generator}, we can deduce that there is a piecewise dominant sequence $\nu\in I^\alpha$ with respect to $\Lam$.

Conversely, suppose there is a piecewise dominant sequence $\nu\in I^\alpha$ with respect to $\Lam$. By Corollary \ref{Pd trace}, we see that $S(\nu)$ is a non-zero element in $\R[\alpha](K)$ which implies $\R[\alpha](K)\neq 0$.
\end{proof}

As an application, we obtain the following criterion for which $\Lam-\alpha$ is a weight of the irreducible highest weight $\mathfrak{g}$-module $L(\Lam)$.

\begin{thm}\label{maincor1} Suppose $L(\Lam)$ is the irreducible highest weight $\mathfrak{g}$-module with highest weight $\Lam$. Then $\Lam-\alpha$ is a weight of $L(\Lam)$ if and only if there is a piecewise dominant sequence $\nu\in I^\alpha$ with respect to $\Lam$. In that case, if  $\nu\in I^\alpha$ is a piecewise dominant sequence with respect to $\Lam$, then $f_{\nu_n}f_{\nu_{n-1}}\cdots f_{\nu_1}v_\Lam\neq 0$ is a nonzero weight vector in $L(\Lam)_{\Lam-\alpha}$, where $f_i$ denotes the Chevalley generator of the enveloping algebra $U(\mathfrak{g})$ for each $i\in I$.
\end{thm}

\begin{proof} The first statement follows from the equality $\dim L(\Lam)_{\Lam-\alpha}=\#\Irr(\R[\alpha])$ (\cite[Theorem 6.2]{KK}) and Theorem \ref{criterion}.

Let $\beta\in Q^+$. For each $i\in I$, let $$\begin{aligned}  F_i^\Lam:\, \Mod(\R[\beta])&\rightarrow \Mod(\R[\beta+\alpha_i]),\\
M&\mapsto \R[\beta+\alpha_i]e(\beta,i)\otimes_{\R[\beta]}M ,
\end{aligned}
$$
be the induction functor introduced in \cite{KK}. Let $[F_i^\Lam]: K(\Proj\R[\beta])\rightarrow K(\Proj\R[\beta+\alpha_i])$ be the induced map on the Grothendieck group of finite dimensional projective modules. Set ${\rm F}_i:=q_i^{1-\<h_i,\Lam-\beta\>}[F_i^\Lam]$. Then by \cite[Theorem 6.2]{KK}, we have
$$
[\R[\alpha]e(\nu)]={\rm F}_{\nu_n}{\rm F}_{\nu_{n-1}}\cdots {\rm F}_{\nu_1} [\textbf{1}_{\Lam}]=f_{\nu_n}f_{\nu_{n-1}}\cdots f_{\nu_1}v_\Lam.
$$ Again, by Theorem \ref{criterion}, if $\nu\in I^\alpha$ is a piecewise dominant sequence with respect to $\Lam$, then $[\R[\alpha]e(\nu)]\neq 0$. Hence, $f_{\nu_n}f_{\nu_{n-1}}\cdots f_{\nu_1}v_\Lam\neq 0$, and it is a nonzero weight vector in $L(\Lam)_{\Lam-\alpha}$.
\end{proof}

%Note that, though we have given some closed formulae for the graded dimension of $\R[\alpha]$ in \cite[Theorem 1.1]{HS}, it is by no means clear from those formulae why piecewise dominant sequences should exist whenever $\R[\alpha]\neq 0$ and why $e(\nu)\neq 0$ in $\R[\alpha]$ for any piecewise dominant sequence $\nu\in I^\alpha$ with respect to $\Lam$.

Suppose that $\R[\alpha]\neq 0$. By Theorem \ref{criterion}, we have a piecewise dominant sequence $\nu$ with respect to $\Lam$. By Corollary \ref{Pd trace}, we know that $S(\nu)+[\RR^\Lam_\alpha,\RR^\Lam_\alpha]$ is a non-zero element in $\bigl(\Tr(\RR^\Lam_\alpha)\bigr)_{d_{\Lam,\alpha}}$.

Note that as $\R[\alpha]$ is a symmetric algebra (\cite{SVV}), we have $Z(\R[\alpha])\cong\bigl(\Tr(\RR^\Lam_\alpha)\bigr)^*\<d_{\Lam,\alpha}\>$. In particular, $\dim\bigl(Z(\R[\alpha])\bigr)_0\cong\dim\bigl(\Tr(\RR^\Lam_\alpha)\bigr)_{d_{\Lam,\alpha}}$. Thus $\dim\bigl(\Tr(\RR^\Lam_\alpha)\bigr)_{d_{\Lam,\alpha}}=1$ if and only if $\dim\bigl(Z(\R[\alpha])\bigr)_0=1$, and if and only if $\R[\beta]$ is indecomposable. The following conjecture is a generalization and refinement of Conjecture \ref{conj2} and \cite[Conjecture 3.33]{SVV}.

\begin{conj}\label{conj1} We have that
$$\bigl(\Tr(\RR^\Lam_\alpha)\bigr)_{d_{\Lam,\alpha}}=K\bigl(S(\nu)+[\RR^\Lam_\alpha,\RR^\Lam_\alpha]\bigr).$$
In particular, $\dim\bigl(\Tr(\RR^\Lam_\alpha)\bigr)_{d_{\Lam,\alpha}}=1$.
\end{conj}

\begin{rem}\label{3rem} 1) By \cite[Remark 3.41]{SVV}, we know that when $\cha K=0$, $\mathfrak{g}$ is symmetric and of finite type, and $\{Q_{ij}(u,v)|i,j\in I\}$ are given as \cite[(11)]{SVV}, the above conjecture holds.

2) Let $\mathfrak{g}$ be of type $A_{\infty}$ or affine type $A_{(e-1)}^{(1)}$ with $e>1$ and $(e,p)=1$, where $p:=\cha K$, and $\{Q_{ij}(u,v)|i,j\in I\}$ are given as \cite[\S3.2.4]{Rou2}. Then after a finite extension of the ground field $K$, each cyclotomic quiver Hecke algebra $\R[\alpha]$ is isomorphic to the block algebra of the cyclotomic Hecke algebra of type $G(\ell,1,n)$ (\cite{BK:GradedKL}) which corresponds to $\alpha$. In this case, the above conjecture holds because $e(\alpha):=\sum_{i\in I^\alpha}e(\bi)$ is a block idempotent of the corresponding cyclotomic Hecke algebra by \cite{LM} and \cite{Br}.

3) By \cite{HuLin} or \cite[Theorem 1.9]{HS2}, we also know that the above conjecture holds whenever $\alpha=\sum_{j=1}^n\alpha_{i_j}$ with $\alpha_{i_1},\cdots, \alpha_{i_n}$ pairwise distinct.
\end{rem}

For any prime number $p>0$, we use $\Z_{(p)}$ to denote the localization of $\Z$ at its maximal ideal $(p)$, and $\hat{\Z}_{(p)}$ to denote the completion of $\Z_{(p)}$ at its unique maximal ideal $p\Z_{(p)}$. Let $\hat{\Q}_{(p)}$ be the fraction field of $\hat{\Z}_{(p)}$.

\begin{lem}\label{plem} Let $K$ be a field with characteristic $\cha K=p>0$, $\Lam\in P^+$ and $\alpha\in Q_n^+$. Suppose that each $Q_{ij}(u,v)$ is defined over $\Z$. If Conjecture \ref{conj1} holds for the cyclotomic quiver Hecke algebra $\R[\alpha](\Q)$, then Conjecture \ref{conj1} holds for the cyclotomic quiver Hecke algebra $\R[\alpha](K)$ too.
\end{lem}

\begin{proof} Applying \cite[Proposition 2.1(d)]{SVV}, we can assume without loss of generality that $K=\mathbb{F}_p$, the finite field with $p$ elements.
Since $\Q\hookrightarrow\hat{\Q}_{(p)}$, Conjecture \ref{conj1} holds for $\R[\alpha](\Q)$ implies that it also holds for $\R[\hat{\Q}_{(p)}]$.
For any $\O\in\{\hat{\Z}_{(p)},K,\hat{\Q}_{(p)}\}$, we set $$
e(\alpha)_{\O}:=\sum_{\bi\in I^\alpha}e(\bi)_{\O}.
$$
which is a central idempotents in $\R[\alpha](\O)$. In particular, $e(\alpha)_{\hat{\Z}_{(p)}}$ is a lift of $e(\alpha)_{K}$ in $\R[\alpha](K)$. Now Conjecture \ref{conj1} holds for $\R[\alpha](\hat{\Q}_{(p)})$ means
$e(\alpha)_{\hat{\Q}_{(p)}}$ is a central primitive idempotent in $\R[\alpha](\hat{\Q}_{(p)})$.

Note that $$\begin{aligned}
&\R[\alpha](K)\cong K\otimes_{\hat{\Z}_{(p)}}\R[\alpha](\hat{\Z}_{(p)})\cong \R[\alpha](\hat{\Z}_{(p)})/\mathfrak{m}\R[\alpha](\hat{\Z}_{(p)}),\\
&\R[\alpha](\hat{\Q}_{(p)})\cong \hat{\Q}_{(p)}\otimes_{\hat{\Z}_{(p)}}\R[\alpha](\hat{\Z}_{(p)}),
\end{aligned}
$$
where $\mathfrak{m}$ is the unique maximal ideal of $\hat{\Z}_{(p)}$. By Corollary \ref{free}, $\R[\alpha](\hat{\Z}_{(p)})$ is free over $\hat{\Z}_{(p)}$.

Suppose that $e(\alpha)_{K}=\overline{e(1)}\oplus\cdots\oplus \overline{e(k)}$ is a decomposition of $e(\alpha)_{K}$ into a direct sum of pairwise orthogonal
(nonzero) central primitive idempotents in $\R[\alpha](K)$. Applying \cite[Proposition 5.22, Theorem 6.7, \S6, Exercise 8]{CR}, for each $1\leq i\leq k$, we can get a lift $e(i)$ in $\R[\alpha](\hat{\Z}_{(p)})$ of $\overline{e(i)}$, such that $\{e(i)|1\leq i\leq k\}$ is a set of pairwise orthogonal central idempotents in $\R[\alpha](\hat{\Z}_{(p)})$ and $\sum_{j=1}^k e(i)=e(\alpha)_{\hat{\Z}_{(p)}}$. Now we have $$
e(\alpha)_{\hat{\Q}_{(p)}}=(1_{\hat{\Q}_{(p)}}\otimes_{\hat{\Z}_{(p)}}e(1))\oplus\cdots\oplus (1_{\hat{\Q}_{(p)}}\otimes_{\hat{\Z}_{(p)}}e(k)),
$$
is a decomposition of $e(\alpha)_{\hat{\Q}_{(p)}}$ into a direct sum of pairwise orthogonal central elements in $\R[\alpha]({\hat{\Q}_{(p)}})$. Moreover, $$
(1_{{\hat{\Q}_{(p)}}}\otimes_{\hat{\Z}_{(p)}}e(j))^2=(1_{{\hat{\Q}_{(p)}}}\otimes_{\hat{\Z}_{(p)}}e(j)),\,\,\forall\,1\leq j\leq k.
$$
By the discussion in the last paragraph, $\R[\alpha](\hat{\Z}_{(p)})$ is a torsion-free $\hat{\Z}_{(p)}$-module. It follows that the canonical
map $\R[\alpha](\hat{\Z}_{(p)})\rightarrow\R[\alpha]({\hat{\Q}_{(p)}})$ is injective. Thus each $1_{{\hat{\Q}_{(p)}}}\otimes_{\hat{\Z}_{(p)}}e(j)$ must be nonzero, hence is a (nonzero) central idempotent in
$\R[\alpha]({\hat{\Q}_{(p)}})$. We get a contradiction to the fact that $e(\alpha)_{{\hat{\Q}_{(p)}}}$ is a central primitive idempotent in $\R[\alpha]({\hat{\Q}_{(p)}})$. This proves the lemma.
\end{proof}

We end this subsection with the following theorem, which is the third main result of this paper.

\begin{thm}\label{4mainthm} Let $K$ be a field of arbitrary characteristic. Suppose that the polynomials $\{Q_{ij}(u,v)|i,j\in I\}$ are given as \cite[\S3.2.4]{Rou2}, $\mathfrak{g}$ is either symmetric and of finite type, or $\mathfrak{g}$ is of type $A_{\infty}$ or affine type $A_{(e-1)}^{(1)}$ with $e>1$.  Then Conjecture \ref{conj1} holds.
\end{thm}

\begin{proof} By assumption, each $Q_{ij}(u,v)$ is defined over $\Z$. If $\cha K=0$, the theorem holds by \cite[Remark 3.41]{SVV}, \cite{BK:GradedKL}, \cite{LM} and \cite{Br} (see Remark \ref{3rem}). Now applying Lemma \ref{plem}, we can deduce that the theorem still holds if
$\cha K>0$.
\end{proof}

\subsection{Relations with crystal basis}

Let $\mathcal{B}=\mathcal{B}(\Lam)$ be the crystal base of the integral highest weight $U_q(\fg)$-module $V(\Lam)$. For each $i\in I$, let $\widetilde{e}_i, \widetilde{f}_i: \mathcal{B}\rightarrow\mathcal{B}\sqcup\{0\}$ be the corresponding Kashiwara operators, let $\veps_i,\varphi_i: \mathcal{B}\rightarrow\Z$ be the
associated functions on $\mathcal{B}$, $\wt: \mathcal{B}\rightarrow P$ be the weight map. In this subsection we try to explain Theorem \ref{maincor1} using the crystal graph. In particular, we shall give a second proof of Theorem \ref{maincor1}.

The following is well-known and will be used in the Lemma \ref{pathlem1} and Lemma \ref{pathlem2}.

\begin{lem}\text{\rm (\cite[\S4.2, (4.3)]{Ka})}\label{wtlem} For each $i\in I$ and $b\in\mathcal{B}$,
$$\varphi_i(b)-\veps_i(b)=\<h_i,\wt(b)\>,\,\,\,\varphi_i(b), \veps_i(b)\geq 0. $$
\end{lem}

\begin{lem}\label{pathlem1} Suppose $L(\Lam)$ is the irreducible highest weight $U_q(\mathfrak{g})$-module with highest weight $\Lam$. \begin{equation}\label{pds} \nu=(\underbrace{\nu^{1},\nu^{1},\cdots,\nu^{1}}_{b_{1}\,copies},\cdots,\underbrace{\nu^{p},\nu^{p},\cdots,\nu^{p}}_{b_{p}\,copies})\in I^\alpha,
\end{equation}
is a piecewise dominant sequence with respect to $\Lam$, such that $\nu^{j}\neq\nu^{j+1},\forall\, 1\leq j<p$, $p,b_1,\cdots,b_p\in\Z^{\geq 1}$ with $\sum_{i=1}^{p}b_i=n$.
Then there is a path in the crystal graph of $\mathcal{B}$ of the following form: \begin{equation}\label{path}
v_\Lam\underbrace{\overset{\nu^1}{\rightarrow}\cdot\overset{\nu^1}{\rightarrow}\cdots\overset{\nu^1}{\rightarrow}}_{\text{$b_1$ copies}}\cdot\cdots\cdot
\underbrace{\overset{\nu^p}{\rightarrow}\cdot\overset{\nu^p}{\rightarrow}\cdots\overset{\nu^p}{\rightarrow}}_{\text{$b_p$ copies}} b .
\end{equation}
We call (\ref{path}) the crystal path associated to the piecewise dominant sequence $\nu=(\nu_1,\cdots,\nu_n)$ and set $b:=b_{\nu}$.
\end{lem}

\begin{proof} Using Lemma \ref{wtlem}, we can calculate $$
\varphi_{\nu^1}(v_\Lam)=\<h_{\nu^1},\wt(v_\Lam)\>+\veps_{\nu^1}(v_\Lam)\geq \<h_{\nu^1},\wt(v_\Lam)\>\geq b_1,
$$
where the last inequality follows from the definition of piecewise dominant sequence. It follows that $b^{(1)}:=\widetilde{f}_{\nu^1}^{b_1}v_\Lam\in\mathcal{B}$. Now we have $\wt(b^{(1)})=\wt(v_\Lam)-b_1\alpha_{\nu^1}=\Lam-b_1\alpha_{\nu^1}$. It follows that $$
\varphi_{\nu^2}(b^{(1)})=\<h_{\nu^2},\Lam-b_1\alpha_{\nu^1}\>+\veps_{\nu^2}(v_\Lam)\geq \<h_{\nu^2},\Lam-b_1\alpha_{\nu^1}\>\geq b_2,
$$ by Lemma \ref{wtlem}.
Hence $b^{(2)}:=\widetilde{f}_{\nu^2}^{b_2}\widetilde{f}_{\nu^1}^{b_1} v_\Lam\in\mathcal{B}$. In general, suppose that $b^{(t)}=\widetilde{f}_{\nu^t}^{b_t}\cdots \widetilde{f}_{\nu^1}^{b_1} v_\Lam\in\mathcal{B}$ is already defined, where $1\leq t<p$. Then we have $\wt(b^{(t)})=\Lam-\sum_{j=1}^{t}b_j\alpha_{\nu^j}$. It follows that $$
\varphi_{\nu^{t+1}}(b^{(t)})=\<h_{\nu^{t+1}},\Lam-\sum_{j=1}^{t}b_j\alpha_{\nu^j}\>+\veps_{\nu^{t+1}}(b^{(t)}))\geq \<h_{\nu^{t+1}},\Lam-\sum_{j=1}^{t}b_j\alpha_{\nu^j}\>\geq b_{t+1},
$$by Lemma \ref{wtlem} again.
Hence $b^{(t+1)}:=\widetilde{f}_{\nu^{t+1}}^{b_{t+1}}\cdots \widetilde{f}_{\nu^1}^{b_1} v_\Lam\in\mathcal{B}$. By an induction on $t$, we get a path in the crystal graph of $\mathcal{B}$ of the form (\ref{path}).
\end{proof}

\begin{lem}\label{pathlem2} Suppose $L(\Lam)$ is the irreducible highest weight $U_q(\mathfrak{g})$-module with highest weight $\Lam$, $b\in\mathcal{B}=\mathcal{B}(\Lam)$ with $\wt(b)=\Lam-\alpha$, $\alpha\in Q_n^+$. Then for any $i\in I$ satisfying $\veps_i(b)>0$, there is a path in the crystal graph of $\mathcal{B}$ of the following form: \begin{equation}\label{path2}
v_\Lam\underbrace{\overset{\nu^1}{\rightarrow}\cdot\overset{\nu^1}{\rightarrow}\cdots\overset{\nu^1}{\rightarrow}}_{\text{$b_1$ copies}}\cdot\cdots\cdot
\underbrace{\overset{\nu^p}{\rightarrow}\cdot\overset{\nu^p}{\rightarrow}\cdots\overset{\nu^p}{\rightarrow}}_{\text{$b_p$ copies}} b ,
\end{equation}
such that $\nu^p=i$ and it is the crystal path associated to the piecewise dominant sequence $\nu\in I^\alpha$ (see (\ref{pds})) with respect to $\Lam$.
\end{lem}

\begin{proof}  Suppose $\Lam-\alpha$ is a weight of $L(\Lam)$. We use induction on $|\alpha|$. For any $b\in\mathcal{B}$ with $\wt(b)=\Lam-\alpha$, we can find $i\in I$ such that $b_0:=\veps_i(b)>0$. Set $b':=\widetilde{e}_i^{b_0}b$. Then $b'\in\mathcal{B}$ and
$\wt(b')=\wt(b)+b_0\alpha_i$. As a result, $$
\<h_i,\wt(b')\>=\<h_i,\wt(b)\>+2b_0=\varphi_i(b)-b_0+2b_0=\varphi_i(b)+b_0\geq b_0,
$$where in the second equality we have used Lemma \ref{wtlem}.
By induction hypothesis, there is a path in the crystal graph of $\mathcal{B}$: $$
v_\Lam\underbrace{\overset{\nu^1}{\rightarrow}\cdot\overset{\nu^1}{\rightarrow}\cdots\overset{\nu^1}{\rightarrow}}_{\text{$b_1$ copies}}\cdot\cdots\cdot
\underbrace{\overset{\nu^{p-1}}{\rightarrow}\cdot\overset{\nu^{p-1}}{\rightarrow}\cdots\overset{\nu^{p-1}}{\rightarrow}}_{\text{$b_{p-1}$ copies}} b' ,
$$
where $\nu^{j}\neq\nu^{j+1},\forall\, 1\leq j<p-1$, $p-1,b_1,\cdots,b_{p-1}\in\Z^{\geq 1}$ with $\sum_{i=1}^{p-1}b_i=n-b_0$, such that $$ \nu=(\underbrace{\nu^{1},\nu^{1},\cdots,\nu^{1}}_{b_{1}\,copies},\cdots,\underbrace{\nu^{p-1},\nu^{p-1},\cdots,\nu^{p-1}}_{b_{p-1}\,copies})\in I^{\alpha-b_0\alpha_i},
$$
is a piecewise dominant sequence with respect to $\Lam$. By construction, $\veps_i(b')=0$. It follows that $\nu^{p-1}\neq i$. Concatenating this path with the path $b'\underbrace{\overset{i}{\rightarrow}\cdot\overset{i}{\rightarrow}}_{\text{$b_0$ copies}}b$ we prove the statement.
\end{proof}

\begin{rem} We remark that the Theorem \ref{maincor1} can be deduced from Lemma \ref{pathlem1} and Lemma \ref{pathlem2}, which also gives a second proof of Theorem \ref{criterion}. Although Lemma \ref{pathlem2} says that for each $b\in\mathcal{B}$, we can always find a crystal path associated to a piecewise dominant sequence, the crystal path of this kind may not be unique.
\end{rem}

For any two piecewise dominant sequence $\mu,\nu\in I^\alpha$, we define $\mu\sim\nu$ if and only if $b_\mu=b_\nu$. In particular, this defines an equivalence relation ``$\sim$'' on the set $\PD$ of piecewise dominant sequences with respect to $\Lam$. Let $\PD\!/\!\sim$ be a set of representatives of all the equivalence classes in $\PD$. We end this paper with the following conjecture.

\begin{conj} Suppose $$\PD\!/\!\sim=\{\mu^{(i)}|1\leq i\leq m\}.
$$ Then the set of elements \begin{equation}\label{basis11}
\bigl\{e(\mu^{(i)})+[\R[\alpha],\R[\alpha]]\bigm|1\leq i\leq m\bigr\}
\end{equation}
forms a $K$-basis of the degree $0$ component $\Tr(\R[\alpha])_0$ of the cocenter of $\R[\alpha]$, and the following set of elements \begin{equation}\label{basis12}
\bigl\{1_\Q\otimes_{\Z}[\R[\alpha]e(\mu^{(i)})]\bigm|1\leq i\leq m\bigr\}
\end{equation}
forms a $\Q$-basis of $\Q\otimes_{\Z}K(\Proj\R[\alpha])$.
\end{conj}


\bigskip
\begin{thebibliography}{2}

\bibitem{APS} {\sc S.~Ariki, E.~Park and L.~Speyer}, {\em Specht modules for quiver Hecke algebras of type $C$}, Publ. Res. Inst. Math. Sci., {\bf 55}(3) (2019), 565--626.

\bibitem{Br}
{\sc J.~Brundan}, {{\em  Centers of degenerate cyclotomic {H}ecke algebras and parabolic category
  {$\mathcal O$}}}, Represent. Theory, {\bf 12} (2008), 236--259.


\bibitem{BK:GradedKL}
{\sc J.~Brundan and A.~Kleshchev},  {\em Blocks of cyclotomic {H}ecke algebras and {K}hovanov-{L}auda algebras}, Invent. Math., {\bf 178}
  (2009), 451--484.

\bibitem{CR}
{\sc C.~W. Curtis and I.~Reiner}, {\em Methods of Representation theory, with applications to finite groups
  and orders}, Vol. I, A Wiley Interscience, 1981.

\bibitem{HuL}
{\sc J.~Hu and X.f.~Liang}, {\em On the structure of cyclotomic nilHecke algebras}, Pac. J. Math., {\bf 296}(1) (2018), 105--139.

\bibitem{HuLin}
{\sc J.~Hu and H.~Lin}, {\em On the center conjecture for the cyclotomic KLR algebras}, preprint, {arXiv:2204.11659}, 2022.

\bibitem{HM}
{\sc J.~Hu and A.~Mathas}, {\em Graded cellular bases for the cyclotomic Khovanov-Lauda-Rouquier algebras of type $A$}, Adv. Math., {\bf 225}(2) (2010), 598--642.


\bibitem{HS}
{\sc J.~Hu, L.~Shi}, {\em Graded dimensions and monomial bases for the cyclotomic quiver Hecke algebras}, preprint, {arXiv:2108.05508}, 2021.

\bibitem{HS2}
{\sc J.~Hu, L.~Shi}, {\em Graded dimensions and monomial bases for the cyclotomic quiver Hecke superalgebras}, preprint, {arXiv:2111.03296 }, 2021.

\bibitem{KK}
{\sc S.~J. Kang and M.~Kashiwara}, {\em Categorification of highest weight modules via Khovanov-Lauda-Rouquier algebras}, Invent. Math., {\bf 190} (2012), 699--742.

\bibitem{Ka}
{\sc M.~Kashiwara}, {\em On crystal bases}, in: Representations of groups (Banff, AB, 1994), CMS Conf. Proc., Vol. 16 (Amer. Math. Soc., Providence, RI), 157--197.

\bibitem{KL1}
{\sc M.~Khovanov and A.D.~Lauda}, {\em A diagrammatic approach to categorification of quantum groups, I}, Represent. Theory, {\bf 13} (2009), 309--347.

\bibitem{KL2}
\leavevmode\vrule height 2pt depth -1.6pt width 23pt,  {\em A diagrammatic approach to categorification of quantum groups, II}, Trans. Amer. Math. Soc., {\bf 363} (2011), 2685--2700.


%\bibitem{Li}
%{\sc G.~Li},  {\em Integral Basis Theorem of cyclotomic Khovanov-Lauda-Rouquier algebras of Type $A$}, J. Alg., {\bf 482} (2017), 1--101.


\bibitem{LM}
{\sc S.~Lyle and A.~Mathas}, {\em Blocks of cyclotomic {H}ecke algebras}, Adv.
  Math., {\bf 216} (2007), 854--878.

\bibitem{MT1}
{\sc A.~Mathas and D.~Tubenhauer}, {\em Subdivision and cellularity for weighted KLRW algebras},
preprint, arXiv:2111.12949, 2021.

\bibitem{MT2}
{\sc A.~Mathas and D.~Tubenhauer}, {\em Cellularity for weighted KLRW algebras of types $B, A^{(2)}, D^{(2)}$},
preprint, arXiv:2201.01998, 2022.


\bibitem{Rou1}
{\sc R.~Rouquier}, {\em $2$-Kac--Moody algebras}, preprint, math.RT/0812.5023v1, 2008.

\bibitem{Rou2}
\leavevmode\vrule height 2pt depth -1.6pt width 23pt, {\em Quiver Hecke algebras and 2-Lie algebras}, Algebr. Colloq. {\bf 19} (2012), 359--410.


\bibitem{SVV}
{\sc P.~Shan, M.~Varagnolo and E.~Vasserot}, {\em On the center of quiver-Hecke algebras}, Duke Math. J., {\bf 166}(6) (2017), 1005--1101.

\bibitem{VV}
{\sc M.~Varagnolo and E.~Vasserot}, {\em Canonical bases and KLR algebras}, J. reine angew. Math., {\bf 659} (2011), 67--100.

\end{thebibliography}
\end{document}